\documentclass[notitlepage,reqno,11pt]{amsart}
\usepackage{latexsym,amssymb, epsfig, amsmath,amsfonts, subfigure,amsthm}

\usepackage{rotating}
\usepackage[toc,page]{appendix}
\usepackage{color}
\usepackage{multirow}
\usepackage{relsize}
\usepackage{microtype}

\usepackage[colorlinks, allcolors=blue]{hyperref}
\usepackage{xcolor}
\usepackage[normalem]{ulem}

\usepackage[margin=1in]{geometry}

\numberwithin{equation}{section}
\newtheorem{lemma}{Lemma}[section]
\newtheorem{theorem}{Theorem}[section]

\theoremstyle{definition}
\newtheorem{assumption}{Assumption}[section]

\newtheorem{remark}{Remark}[section]

\newtheorem{example}{Example}[section]

\newlength{\defbaselineskip}
\newcommand{\setlinespacing}[1]%
           {\setlength{\baselineskip}{#1 \defbaselineskip}}

\newcommand{\RR}{{\mathbb R}}
\newcommand{\ZZ}{{\mathbb Z}}
\newcommand{\NN}{{\mathbb N}}

\newcommand{\ra}{\rightarrow}

\newcommand{\beql}[1]{\begin{equation}\label{#1}}
\newcommand{\eeq}{\end{equation}}

\newcommand{\beqal}[1]{\begin{eqnarray}\label{#1}}
\newcommand{\eeqa}{\end{eqnarray}}
\newcommand{\beq}{\begin{displaymath}}
\newcommand{\eeqno}{\end{displaymath}}
\newcommand{\bali}[1]{\begin{align}\label{#1}}
\newcommand{\eali}{\begin{align}}
\newcommand{\balino}{\begin{align*}}
\newcommand{\ealino}{\begin{align*}}

\newcommand{\qandq}{\quad\mbox{and}\quad}

\newcommand{\qforq}{\quad\mbox{for}\quad}

\newcommand{\qasq}{\quad\mbox{as}\quad}

\newcommand{\qforallq}{\quad\mbox{for all}\quad}

\newcommand{\baa}{\begin{eqnarray*}}
\newcommand{\eaa}{\end{eqnarray*}}

\DeclareMathOperator{\tr}{tr}

\def\wtZ{\widetilde Z}

\let\oldtocsection=\tocsection
\let\oldtocsubsection=\tocsubsection
\let\oldtocsubsubsection=\tocsubsubsection
\renewcommand{\tocsection}[2]{\hspace{0em}\oldtocsection{#1}{#2}}
\renewcommand{\tocsubsection}[2]{\hspace{1em}\oldtocsubsection{#1}{#2}}
\renewcommand{\tocsubsubsection}[2]{\hspace{2em}\oldtocsubsubsection{#1}{#2}}
\def\btau{{\mbox{\boldmath$\tau$}}}

\def\eps{\epsilon}

\def\beq{\begin{equation}}
\def\eeq{\end{equation}}

\def\beq{\begin{equation}}
\def\eeq{\end{equation}}

\newcommand{\ttl}{\Large Birth and Death Processes\\[5pt] in Interactive Random Environments}

\newcommand{\ttls}{\large Birth and Death Processes in Interactive Random Environments}

\begin{document}

\title[\ttls]{\ttl}

\author{Guodong Pang}
\address{Department of Computational Applied Mathematics and Operations Research,
George R. Brown College of Engineering,
Rice University, Houston, TX}
\email{gdpang@rice.edu}

\author{Andrey Sarantsev}
\address{Department of Mathematics and Statistics, University of Nevada, Reno, NV}
\email{asarantsev@unr.edu}

\author{Yuri Suhov}
\address{
Department of Pure Mathematics and Mathematical Statistics, University of Cambridge; Department of Mathematics, Pennsylvania State University}
\email{yms@statslab.cam.ac.uk; ims14@psu.edu}

\def\cre{\color{red}} \def\cbr{\color{brown}}

\begin{abstract} 
This paper studies birth and death processes in interactive random environments where the birth and death rates and the dynamics 
of the state of the environment are dependent on each other. Two models of a random environment are considered:
a continuous-time Markov chain (finite or countably infinite) and a reflected (jump) diffusion process. The background is determined by a  joint Markov process carrying a specific interactive mechanism, with an explicit invariant measure whose structure is similar to a product form. We discuss a number of queueing and population-growth models and establish conditions under which the above-mentioned invariant measure 
can be derived.  

Next, an analysis of the rate of convergence to stationarity is performed for the models under consideration. We consider two settings leading to either an exponential or a polynomial convergence rate. In both cases we assume that the underlying environmental 
Markov process has an 
exponential rate of convergence, but the convergence rate of the joint Markov process is determined by certain conditions on the 
birth and death rates. To prove these results a coupling method turns out to be useful. 
\end{abstract}


\maketitle

\begin{center}
{\it Contribution to the QUESTA Special Issue in the honor of \\ Professor Masakiyo Miyazawa's 75th Birthday} 
\end{center}

\thispagestyle{empty}

\allowdisplaybreaks

\section{Introduction} \label{sec-intro}

Birth-death processes are fundamental stochastic models in applied probability,  particularly in queueing. 
Birth-death processes in random environments have been extensively studied \cite{torrez1978birth,torrez1979calculating,cogburn1981birth,cornez1987birth} and used in various applications (e.g., 
in queueing \cite{krenzler2015loss,krenzler2016jackson,gannon2016random}, inventory \cite{otten2020queues},  population
dynamics  \cite{bacaer2014linear} and epidemiology \cite{prodhomme2021large}). Most of the studies have been
 focused on models where the transitions of birth and death processes are affected by the environment   but
not {\it vice versa} (see, e.g., \cite{torrez1978birth,torrez1979calculating,cogburn1981birth}). 

However, in applications, the influence of the birth-death processes and the environment can be in both directions.  For example,  
service systems can be often modeled as multi-server queues where customer arrivals may depend on performance rating. 
The joint ({\it queue, rating}) dynamics may be modeled as a Markov process where ratings depend on the service quality indicated by congestion (i.e., the size of the queue). The  
population growth can also be modeled as a birth-death process in a random environment where the environmental changes are
influenced by the population size (say, in the case of an overpopulation).  In epidemiology,  the infection rate may depend on a 
moving environment, while the dynamics of the environmental state is determined by the number of infected individuals (say, via a 
specific vaccination or lockdown intervention policy).

In this paper, we study birth--death processes in an interactive random environment evolving in such a way that 
the joint Markov process has a product-form type invariant measure. Consider a birth-death process 
$N = \{N(t),\, t \ge 0\}$ on $\NN=\{0, 1, 2, \ldots\}$ with birth rates $\lambda_n(z)$ (for the jump $n\to n+1$, $n\geq 0$)
and death rates $\mu_n(z)$ (for the jump $n\to n-1$, $n\geq 1$) 
depending on an {\it environment variable} $z$ taking values in an {\it environment space} $\mathcal Z$.
 The environmental variable, in turn, 
evolves as a continuous-time Markov process $Z = \{Z(t),\, t \ge 0\}$, with transition functions depending
on the current state $N(t)$ of  the birth-death process.
\footnote{Formally speaking, we deal with a {\it family}
of birth-death processes depending on the parameter $z\in{\mathcal Z}$ and a {\it family} of environmental processes 
depending on the parameter $n\in{\mathbb N}$.} A combined Markov process 
$(N, Z) = \{(N(t), Z(t)),\, t \ge 0\}$, with states $(n,z)$, is called a {\it birth-death process in an interactive 
random environment}. 

The generator  $\mathcal L$ of the combined process $(N,Z)$ is given by 
\begin{equation}
\label{eq:generator}
{\mathcal L}f(n, z) = {\mathcal M}_zf(n, z) +  r_n(z)^{-1}{\mathcal A}_nf(n, z).
\end{equation}
Here $\mathcal M_z$ is the generator of the birth-death process with a fixed environment variable $z$ and $\mathcal A_n$ is  
the generator of the environment process for a fixed birth-death value $n$. 
The parameter $r_n(z)$ is the  cumulative (product) birth-death rate ratio given in  \eqref{eqn-pi-n-BD}. 
See further discussions on the joint generator in Remark \ref{rem-generator}.

For the  process $Z$, we consider two  models: 
\begin{itemize}
\item[(i)] The environment space $\mathcal Z$ is   at most countable, and the environment process is a 
continuous-time Markov chain with a generating matrix  ${\mathcal A}_n$ depending on $n$, the state of 
 process $N(t)$; 
we call this model a {\it jump environment}. 
\item[(ii)] The environment space $\mathcal Z$ is a domain in $\mathbb R^d$, and the environment process is a 
 reflected jump diffusion in this domain, with  variable drift vector  $b_n(z)$, diffusion matrix  $\sigma_n(z)$
and jump measures $\varpi_n(z,\;\;)$ 
dependent on the state $n$ of process $N(t)$;  this is called a {\it diffusive environment}.\footnote{The domain 
and the reflection type may also depend on $n$; a general type of dependence is encrypted in the symbol ${\mathcal A}_n$
for the generator of the environmental diffusion.}   
\end{itemize}

In this article, we study the long-time behavior of these models: 

\begin{itemize}
\item[(a)]  existence and uniqueness of a stationary distribution, that is, a probability distribution $\pi$ on the product space 
$\mathcal X = \NN\times \mathcal{Z}$ such that if $(N(0), Z(0)) \sim \pi$, then $(N(t), Z(t)) \sim \pi$ for all $t \ge 0$, and 
an explicit form of $\pi$; 
\item[(b)]  convergence $(N(t), Z(t)) \to \pi$ as $t \to \infty$ in the total variation distance, and the rate of this 
convergence (exponential or polynomial). 
\end{itemize}

We adapt and generalize the methods of our previous article \cite{PSBS} devoted to M/M/1 queues in an interactive random 
environment. We identify conditions on the birth and death rates and the underlying Markov process under which the rate of 
convergence can be either exponential or polynomial (Theorems \ref{thm-RC-exp}, \ref{thm-RC-exp2} and \ref{thm:polynomial}). 

\medskip

We discuss a few examples that are of interest on their own. For example, we have studied infinite-server queues with the arrival and/or service rates being an RBM or reflected Ornstein--Ulenbeck diffusion (Examples \ref{exm:RBM-arrival}--\ref{example-MMinfty-compact}). We have also discussed the finite-server queues (infinite-waiting space, blocking/loss model or with abandonment) where the arrival, service and/or abandonment rates are an RBM or reflected diffusion in Examples \ref{example-MMK-diff} and \ref{example-MMK+M-diff}. Another example is the population growth model in biology with the growth and death rates dependent on the environment (see Example \ref{example-population} in a jump environment and Example \ref{example-population-diff} with the rates being a three-dimensional RBM in an orthant). 
We have also briefly discussed how the population growth model can be extended to study growth stocks in finance in Examples \ref{example-growthstock}  and  \ref{example-population-diff}. In all these models, we discuss how the conditions for the existence of stationary distributions are verified and provide the explicit expressions for the invariant measures. 

\subsection{Literature review}
Birth-death processes in random environments have been widely studied, see, e.g., \cite{torrez1978birth,torrez1979calculating,cogburn1981birth}. In these models, the birth and death rates are affected by the environments.  Economou \cite{economou2005generalized} studied continuous-time Markov chains (CTMC) in random environments where not only the the transitions rates of the CTMC depend on the environment, but also a change in the environment can trigger an immediate transition of the CTMC. He identified conditions under which a (generalized) product form stationary distribution may exist. In \cite{cogburn1980markov}, more general Markov chains in random environments are studied, where the transition probabilities of the Markov chains are affected by the environments. Baca{\"e}r and Ed-Darraz \cite{bacaer2014linear} studied a linear birth-death process in a finite-state random environment with biology applications, and derived the probability of extinction. 
In all these studies, the interaction with the environment is one-sided, that is, the dynamics of the environment is not influenced by the state of the Markov chains. 

Cornez \cite{cornez1987birth} first studied birth-death processes in random environment with feedback (that is, feedback to the environment process from the state of the birth-death process) and provided sufficient conditions under which the birth-death process component goes extinct or not. In that model, the environment process takes values in a general measurable space, and no explicit stationary distribution is derived. 
In \cite{krenzler2015loss}, loss queues with interactive (Markov jump) random environments are considered, and a product-form steady state distribution of the joint queueing-environment process is derived which results in a strong insensitivity property. 
In \cite{krenzler2016jackson}, the authors consider Jackson networks with interactive (Markov jump) random environments, where customers departing from the network may enforce the environment to jump immediately. 
In \cite{otten2020queues}, single server queues with state dependent arrival and service rates which are also interactively affected by a Markov jump environment are studied, and both cases of an explicit product-form (separable) steady state distribution and of a non-separable steady-state distribution 
are considered. In \cite{das2016constructions}, another construction is provided for  Markov processes in interactive random environments (pure Markov jump process) that allows simultaneous transitions for the Markov chain and environment states, for which a product form invariant measure is derived and applications to queueing and neural avalanches are discussed.
We note the main differences in the construction of the joint Markov process in \cite{krenzler2015loss,krenzler2016jackson,otten2020queues} 
from our paper: they allow simultaneous changes in the queueing and environment states, while our construction does not. 
Moreover, the environments in those papers are only a Markov jump process. 

We also refer to \cite{gannon2016random}, where a random walk interacting with a random environment of a Jackson/Gordon-Newell network 
is considered, and an explicit stationary distribution of a product-form type is derived.  In \cite{prodhomme2021large}, an epidemic SIS model 
in an interactive switching environment is studied, where the infection and recovery rates depend on a finite-state Markov jump process whose 
transitions also depend on the number of infectives.  Large population scaling limits and the associated long-time behaviors are studied. 

Our work generalizes the previous work in \cite{PSBS}, where an M/M/1 queue in an interactive random environment is studied, with both 
jump and diffusive environments. 
The models considered in this paper are more general, and a few new stochastic models are introduced as discussed above.
In addition to exponential rate of convergence, we also establish polynomial rate of convergence to stationarity. 

This paper also contributes to the understanding of rate of convergence of birth-death processes. 
Lindvall \cite{lindvall1979note} developed the coupling approach to estimate the exponential rate of convergence for birth-death 
processes, which we follow and further develop for our model. 
Van Doorn \cite{van1985conditions} identified conditions on the birth and death rates under which the chain is exponentially ergodic 
by investigating the spectral representation of the transition probabilities, and bounds on the decay parameter were also established. 
 Zeifman \cite{zeifman1991some}  used  methods  of differential equations to derive explicit
estimates for the rate of convergence for birth-death processes; this was  subsequently applied to some queueing 
examples, including the 
M/M/$K$ and M/M/$K/0$ models.  In  \cite{zeifman1995upper}, this approach has been
extended to nonhomogeneous birth-death 
processes for which upper and lower bounds on the rate of convergence were derived;  consequently, a number of  
queueing examples have been studied, including M$_t$/M$_t$/$K$ and M$_t$/M$_t$/$K$/0. 
Van Doorn and Zeifman \cite{van2009speed} study the rate of convergence of the Erlang loss system. 
We also refer to 
\cite{van2010bounds, van2011rate,zeifman2017convergence} for further studies on the related topics and queueing models. 
 
\subsection{Organization of the article} 
In \textsc{Section}~\ref{sec-jump}, we state the model and results for the jump environment. 
In \textsc{Section}~\ref{sec-diffusive}, we do the same for the diffusive environment. 
A few examples are provided in both sections. 
In \textsc{Section}~\ref{sec-conv-rate}, we state and prove the results on the exponential rates of convergence to stationarity. 
In \textsc{Section}~\ref{sec-1/t}, we do the same for polynomial convergence. \textsc{Appendix} A contains proofs of results from \textsc{Sections}~\ref{sec-jump} and~\ref{sec-diffusive}. \textsc{Appendix} B provides a technical comparison lemma from \cite{PSBS}.

\subsection{Notation} We let $\mathbb N = \{0, 1, 2, \ldots\}$ and $\mathbb R_+ := [0, \infty)$ be the sets of all nonnegative integer and 
real numbers, respectively. Let  $\mathbb{R}^d$ be the space of $d$-dimensional real numbers. 
The {\it total variation} distance between two probability measures $P$ and $Q$ on the same space $\mathcal E$ 
is defined as 
\begin{equation}\label{eq:TVD}\|P - Q\|_{\mathrm{TV}} = \sup_{A \subseteq \mathcal E}|P(A) - Q(A)|\,.\end{equation}
 The space of twice continuously 
differentiable functions on the space $\mathcal E$ is denoted by $C^2(\mathcal E)$.
 The space of bounded twice continuously 
differentiable functions with bounded first and second derivatives on the space $\mathcal E$ is denoted by $C^2_b(\mathcal E)$. 

\section{Jump Environment} \label{sec-jump} 

\subsection{Model construction} Consider a birth-death process in an interactive jump environment described as follows. 
Let $\mathcal{Z}$ be a finite or countable state space for the environment.
We define a two-component Markov process $(N, Z)$ taking values in the countable state space 
$\NN\times \mathcal{Z}$, with the following generator matrix 
$\mathbf{R} = \big( R[(n,z), (n',z')]\big)$: 
\begin{equation} \label{eqn-R-BD}
\begin{split}
& R[(n,z), (n+1,z)] = \lambda_n(z), \quad R[(n,z), (n-1,z)] =  \mu_n(z), \\
& R[(n,z), (n,z')] = r_n(z)^{-1}  \tau_n(z,z') \\
& R[(n,z), (n',z')] = 0, \quad n\neq n', \quad z\neq z', 
\end{split}
\end{equation}
where  for each $z \in \mathcal{Z}$,  
\begin{equation}\label{eqn-pi-n-BD} 
r_n(z)  =  \prod_{k=1}^{n}\frac{\lambda_{k-1}(z)}{\mu_k(z)}\, \qforq n \ge 1, \qandq r_0(z) \equiv 1,  
\end{equation}
and 
$\mathbf{T}_n = (\tau_n(z,z'))_{z,z'\in \mathcal{Z}} $ is the generator for an irreducible continuous-time Markov chain on $\mathcal{Z}$ (see e.g., \cite{meyn1993stabilityII}). Here $N = \{N(t): t \ge 0\}$ represents the dynamics of the birth-death process, taking values in $\NN$, and $Z = \{Z(t): t\ge 0\}$  indicates the evolution of the environment. 
It can be easily checked that for each $z \in \mathcal Z$, we have the detailed balance equations for the birth-death process $N(t)$:
\begin{align} 
 & \kappa_n(z)  \big( \lambda_n(z) + \mu_n(z) \big) = \kappa_{n-1}(z) \lambda_{n-1}(z) +  \kappa_{n+1}(z)  \mu_{n+1}(z), \quad n \ge 1,\label{eqn-N-balance-n}  \\
 & \kappa_0(z) \lambda_0(z)=\kappa_1(z)\mu_1(z),  \label{eqn-N-balance-0}
 \end{align}
 where
 \begin{align}
\label{eqn-pi-0-BD}
\begin{split} 
 \kappa_n(z) & =  \kappa_0(z) \prod_{k=1}^{n}\frac{\lambda_{k-1}(z)}{\mu_k(z)}\,,\quad n \ge 1\,, 
\\
\kappa_0(z) & = \left[1 + \sum_{j=1}^\infty \prod_{k=1}^j \frac{\lambda_{k-1}(z)}{\mu_k(z)}\right]^{-1}\,.  
\end{split}
\end{align}
For the quantity $ \kappa_0(z) $ to be well defined, we make the following assumption.
\begin{assumption}\label{as-r0}
For each $z \in  \mathcal{Z}$,
 $\lambda_n(z)$ and $\mu_n(z)$ are positive such that 
\begin{align*}
\sum_{j=1}^\infty  r_j(z) = \sum_{j=1}^\infty \prod_{k=1}^j \frac{\lambda_{k-1}(z)}{\mu_k(z)} < \infty. 
\end{align*}
\end{assumption}
Observe that the detailed balance equations in \eqref{eqn-N-balance-n} and \eqref{eqn-N-balance-0} also hold by replacing 
$\kappa_n(z)$ with $r_n(z)$. This is not surprising: $\kappa_n$ are a normalized $r_n$ with $\sum_{n=0}^\infty \kappa_n=1$. 
%We refer to this matrix $\mathbf{T}_n$ as the \emph{nominal jump intensity matrix}  for the birth-death process in state $n$. 
When the environment is in state $z$, the birth rate of $N(t)$ in the state $n$ is $\lambda_n(z)$ while the death rate  is $\mu_n(z)$. When the birth-death process is in state $n$, the transition of the environment from state $z$ to  $z'$ occurs at the rate $\tau_n(z,z')/r_n(z)$. Note that the last equation in \eqref{eqn-R-BD} forbids simultaneous jumps for  $N$ and $Z$. 
The pair $(N,Z)$ yields a  Markov process in the state space $\NN\times \mathcal{Z}$ with the generator $\mathbf{R}$. We denote its transition kernel by $P^t((n,z), \cdot)$. 

\subsection{Main results on the stationary distribution} We make the following assumption on the matrix $\mathbf{T}_n$. 

\begin{assumption}  \label{as-Tn}
There exists a function $v: \mathcal{Z} \ra \RR_+$ satisfying 
\beql{as-T-v}
v(z)\sum_{z'\in \mathcal{Z}} \tau_n(z,z') = \sum_{z'\in \mathcal{Z}} v(z') \tau_n(z',z)\,, \quad \mbox{for all}\quad n\in \NN,\, z \in \mathcal{Z};
\eeq
and 
\beql{XiFinite}
\Xi :=  \sum_{(n,z)} r_n(z) v(z) = \sum_{(n,z)}\prod_{k=1}^{n}\frac{\lambda_{k-1}(z)}{\mu_k(z)}  v(z) <\infty.
\eeq
\end{assumption}

\begin{theorem} \label{thm-BD}
Under Assumptions \ref{as-r0} and \ref{as-Tn}, the Markov process $(N,Z)$ $\NN\times \mathcal{Z}$  is irreducible, aperiodic, and positive recurrent. It  has a unique invariant probability measure
\beql{eqn-BD-pi}
\pi(n,z) := \eta(n,z)/\Xi,  \quad \forall (n,z) \in \NN\times \mathcal{Z},
\eeq
with 
\beql{eqn-BD-eta}
\eta(n,z) := r_n(z) v(z) , \quad \forall (n,z) \in \NN\times \mathcal{Z},
\eeq
The transition kernel converges to this invariant measure in the total variation distance:
\begin{equation}
\label{eq:ergodic-discrete}
\|P^t((n,z), \cdot) - \pi(\cdot)\|_{\mathrm{TV}} \to 0 \qasq  t \to \infty, \qforallq (n,z) \in \NN\times \mathcal{Z}.
\end{equation}
\end{theorem}

\begin{remark}
Observe that the function $v$ in \eqref{as-T-v} is independent of $n$, which is crucial to the product-form of the invariant measure $\pi$.
The condition in \eqref{as-T-v} will hold if $\tau_n(z,z') $ takes the form $\tau_n(z,z') = \beta_n \tau(z, z')$ for some constant $\beta_n>0$. However, one can construct examples of $\tau_n(z,z')$ of more complicated forms that still guarantee the existence of a function $v$ satisfying \eqref{as-T-v} (see, e.g., Examples 2.1 and 2.2 in \cite{PSBS}). 
We also refer to Section \ref{sec-diff-gen-n} for discussions on the diffusive setting. 
(Constructions similar to those from Section \ref{sec-diff-gen-n} can be done in the discrete setting too.) 
\hfill $\Box$
\end{remark}

\subsection{Examples} 
We start with classic queues: M/M/1, M/M/$\infty$, M/M/$K$, M/M/$K/0$.  
In Examples \ref{MM1}-\ref{example-MMK+M}, the 
parameter $\lambda (z)$ takes nonnegative values while $\mu (z)$ and $\gamma (z)$ are strictly positive: $0\leq\lambda (z)<\infty$ 
and $0<\mu (z),\gamma (z)<\infty$. In all examples in this section, as long as condition \eqref{XiFinite} is fulfilled, there exists a finite 
invariant measure $\eta$ for process $(N,Z)$
on ${\mathbb N}\times{\mathcal Z}$. The unique invatriant probability distribution $\pi (n,z)$ is then obtained by normalization.

\begin{example}\label{MM1}
The M/M/1 queue was already studied in \cite{PSBS}: here $\lambda_n(z) = \lambda(z)$ and $\mu_n(z) = \mu(z)$ for $n\ge 0$ and $z \in \mathcal{Z}$. Hence,
$r_n(z) =  \rho(z)^n$  where $ \rho(z) := \lambda(z)/\mu(z)$ is the traffic intensity satisfying $\rho(z)\in (0,1)$  for all $z\in \mathcal{Z}$. Also, $\kappa_0(z) =  1-\rho(z)$, and 
$$
\kappa_n(z) = \kappa_0(z) \rho(z)^n =(  1-\rho(z)) \rho(z)^n,\quad n\ge 1, \quad z \in \mathcal{Z}.
$$
Condition \eqref{XiFinite} means that, as in \cite[Assumption 2.2]{PSBS},
$$
\Xi =\sum_{(n,z)} \rho(z)^n v(z) = \sum_z \frac{v(z)}{1-\rho(z)}<\infty. 
$$
Then  $\eta(n,z) = \rho(z)^n v(z)$ yields a finite invariant measure for process $(N,Z)$.
\end{example}

\begin{example} For an M/M/$\infty$ queue, $\lambda_n(z) = \lambda(z)$ and $\mu_n(z) = n \mu(z)$ for all $n$ and $z\in \mathcal{Z}$. Hence, 
$$
r_n(z) = \frac{ \rho(z)^n}{n!}, \quad \mbox{where}\quad \rho(z) := \frac{\lambda(z)}{\mu(z)} \in (0,\infty)
$$
is the offered load. The sum of all these $r_n(z)$ is $\sum_{n=0}^\infty r_n(z) = e^{\rho(z)}$. Thus, after normalizing,
we get
$$
\kappa_0(z) = e^{-\rho(z)},,\qandq \kappa_n(z) = e^{-\rho(z)} \frac{\rho(z)^n}{n!}\,,\quad n\ge 1, \quad z \in \mathcal{Z}. 
$$
Condition \eqref{XiFinite} takes the following form
$$
\Xi= \sum_{(n,z)} \frac{\rho(z)^n}{n!} v(z) = \sum_z e^{\rho(z)} v(z)<\infty.
$$ Consequently, $\eta(n,z) =\rho(z)^nv(z)/n!$ yields a finite invariant measure for $(N,Z)$, 
and $\pi (n,z)=\eta(n,z)/\Xi$ is a unique invariant probability distribution.
\end{example}

\begin{example} \label{example-MMK-jump}
Let us fix an integer $K\geq 1$. 
For an M/M/$K$ queue, $\lambda_n(z) = \lambda(z)$ and $\mu_n(z) = \mu(z) (n\wedge K)$ for all $n\ge 0$ and $z\in \mathcal{Z}$.
 Hence, 
\begin{equation}\label{eqn-rn-MMK}
r_n(z) = \begin{cases} 
\frac{\rho(z)^{n}}{n!}\,, & \qforq n <K, \\
\frac{\rho(z)^n }{K!K^{n-K}}\,, & \qforq n \ge K, 
 \end{cases}
\end{equation} 
where $\rho(z) = \lambda(z)/\mu(z)$ is the offered load, and $\varrho(z) = \rho(z)/K$ is the traffic intensity. Also,  
\begin{align*}
\kappa_0(z) &=  \left( \sum_{n=0}^{K-1} \frac{\rho(z)^n}{n!} + \frac{\rho(z)^K}{K!} \frac{1}{1- \rho(z)}\right)^{-1}; \\
\kappa_n(z) &= \begin{cases}
\kappa_0(z) \frac{(\rho(z))^n}{n!}, \quad & n <K; \\
\kappa_0(z) \frac{(\rho(z))^n }{K! K^{n-K}}, \quad & n\ge  K.
\end{cases}
\end{align*}
We assume that the traffic intensity $\varrho(z)= \rho(z)/K<1$ for all $z$. 
Condition \eqref{XiFinite} now means that 
$$
\Xi = \sum_{(n,z)} r_n(z) v(z) = \sum_z \left(\sum_{n=0}^{K-1}\frac{\rho(z)^{n}}{n!} +  \sum_{n=K}^\infty \frac{\rho(z)^n }{K!K^{n-K}} \right) v(z)<\infty.
$$
Then the formula
$$
\eta(n,z) = \begin{cases} 
\frac{\rho(z)^{n}}{n!} v(z)\,, & \qforq n <K, \\
\frac{\rho(z)^n }{K!K^{n-K}} v(z)\,, & \qforq n \ge K.
 \end{cases}
$$
gives a finite invariant measure for $(N,Z)$.

It is clear that the M/M/1 model in \cite{PSBS} is a special case of this model. 
\end{example}

\begin{example} \label{example-MMK0}
For an M/M/$K$/$0$ queue,  $\lambda_n(z) = \lambda(z){\mathbf 1}(0\leq n< K)$, and $\mu_n(z) = n\mu(z)$ for $0 \le n \le K$ and $z\in \mathcal{Z}$.  (The form of rates $\mu_n(z)$ for $>K$ is chosen for convenience.)
Hence, $r_n(z)= \rho(z)^{n}/n!$ where $\rho(z) 
= \lambda(z)/\mu(z)$, $n=0,\ldots ,K$, $z\in \mathcal{Z}$. Next, 
$$\kappa_0(z)= \left( \sum_{n=0}^{K} \frac{\rho(z)^n}{n!} \right)^{-1},\quad\hbox{and}\quad 
\kappa_n(z)= \kappa_0(z) \frac{(\rho(z))^n}{n!},\quad n = 1, \ldots, K.$$
Condition \eqref{XiFinite} is written as
$$
\Xi = \sum_z \left(\sum_{n=0}^{K}\frac{\rho(z)^{n}}{n!}  \right) v(z)<\infty.
$$
A finite invariant measure  for $(N,Z)$ is given by $\eta(n,z) = \rho(z)^{n}v(z){\mathbf 1}(0\leq n\leq K)/n!$.

 A similar construction works for an M/M/$K/l$ model, where $\lambda_n(z)=\lambda (z){\mathbf 1}(0\leq n\leq K+\ell)$
and $\mu_n(z)=(n\wedge K)\mu(z)$. 
\end{example}

\begin{example} \label{example-MMK+M}
For an M/M/$K$+M queue, $\lambda_n(z) = \lambda(z)$, $\mu_n(z) = \mu(z) (n \wedge K) + \gamma(z) (n- K)^+$ for 
$n\ge 0$ and $z \in \mathcal{Z}$. Here,  the rates $\lambda(z), \mu(z)$ and $\gamma(z)$ represent the arrival, service and abandonment 
rates, respectively. Hence,
\begin{equation}\label{eqn-rn-MMK+M}
r_n(z)= \begin{cases} 
\frac{\rho(z)^{n}}{n!}, & \qforq n <K, \\
\frac{\rho(z)^K \beta(z)^{n-K} }{ K! (n-K)!}, & \qforq n \ge K, 
\end{cases}
\end{equation} 
where $\rho(z) =\lambda(z)/\mu(z)$ is the offered load, $\varrho(z)= \rho(z)/K$ is the traffic intensity, and $\beta(z) = \lambda(z)/\gamma(z)$.
For this model,  the traffic intensity $\varrho(z)$ is allowed to take any positive value, less than 1 (underloaded), equal to 1 (critically loaded) or larger than 1 (overloaded). 
In this model we have
\begin{align*}
\kappa_0(z) &=  \left( \sum_{n=0}^{K-1} \frac{\rho(z)^n}{n!} + \frac{\rho(z)^K}{K!} e^{\beta(z) }\right)^{-1}; \\
\kappa_n(z) &= \begin{cases}
\kappa_0(z) \frac{(\rho(z))^n}{n!}, \quad & n <K; \\
\kappa_0(z) \frac{\rho(z)^K \beta(z)^{n-K} }{ K! (n-K)!}, \quad & n\ge  K.
\end{cases}
\end{align*}
Condition \eqref{XiFinite} reads
$$
\Xi = \sum_{(n,z)} r_n(z) v(z) = \sum_z \left(\sum_{n=0}^{K-1} \frac{\rho(z)^n}{n!} + \frac{\rho(z)^K}{K!} e^{\beta(z) } \right) v(z)<\infty.
$$
A finite invariant measure has the form 
$$
\eta(n,z) = \begin{cases} 
\frac{\rho(z)^{n}}{n!} v(z), & \qforq n <K, \\
\frac{\rho(z)^K \beta(z)^{n-K} }{ K! (n-K)!}  v(z), & \qforq n \ge K.
 \end{cases}
$$
\end{example}

Example \ref{example-population} below emerges in biological reproduction and population growth. 

\begin{example} \label{example-population}
Here we consider a linear growth model with immigration (\cite[Example 6.4]{ross2019introduction}): for each $z \in \mathcal{Z}$,  $\lambda_n(z) = n \lambda(z) + \theta(z)$ for $n\ge 0$ and $\mu_n(z) = n \mu(z)$ for $n\ge 1$. 
In this model each individual in the population gives birth at a rate  $\lambda(z)>0$;  in addition, there is an exponential rate of growth of the population $\theta(z)>0$ due to an external source (immigration). The death rate is given by $n \mu(z)$. 
Hence, 
\begin{align}\label{eqn-rn-population}
\begin{split}
r_n(z) &= \frac{\prod_{k=1}^n ((k-1)\lambda(z) + \theta(z)) }{n! \mu(z)^n} = \frac{\rho(z)^n}{n!} \prod_{k=1}^n \left[k-1 + \frac{\theta(z)}{\lambda(z)}\right] \\ & \qforq n \ge 0, \, z \in \mathcal{Z}, 
\end{split}
\end{align} 
where $\rho(z) = \lambda(z)/\mu(z) \in (0,\infty)$.  We also have 
\begin{align*}
\kappa_0(z) &= \left( 1+ \sum_{j=1}^\infty \frac{\rho(z)^j}{j!} \prod_{k=1}^j ( (k-1)+ \theta(z)/\lambda(z)) \right)^{-1},\\
\kappa_n(z) &= \kappa_0(z)  \frac{\rho(z)^n}{n!} \prod_{k=1}^n ((k-1)+ \theta(z)/\lambda(z)), \quad n \ge 1. 
\end{align*}
Condition \eqref{XiFinite} in this case is
$$
\Xi = \sum_{(n,z)} r_n(z) v(z) = \sum_{(n,z)} \frac{\rho(z)^n}{n!} \prod_{k=1}^n (k-1 + \theta(z)/\lambda(z)) v(z) <\infty. 
$$
A finite invariant measure on ${\mathbb N}\times{\mathcal Z}$ has the form 
$$
\eta(n,z) = r_n(z) v(z) = \frac{\rho(z)^n}{n!} \prod_{k=1}^n ( (k-1)+ \theta(z)/\lambda(z)) v(z), \qforq n \ge 0, \, z \in \mathcal{Z}. 
$$
\end{example} 

\begin{example} \label{example-growthstock}
The model from Example \ref{example-population} can be modified to model growth stocks such as Internet or biotech as proposed in \cite{kou2003modeling}. 
In that setting, the parameters $\lambda(z)$ and $\mu(z)$ represent the instantaneous appreciation and depreciation of the stock price due to market fluctuations, and the parameter $\theta(z)\ge 0$ represents the rate of increase in the stock price due to non-market factors such as the effect of additional shares via public offering. One can also include in the death rate an additional external effect parameter, that is, $\mu_n(z) = n\mu(z)+\vartheta(z)$, where $\vartheta(z)\ge 0$ captures the rate of decrease in the stock price due to non-market factors such as dividend payments (for most growth stocks, dividends are zero). 
In this case, we have
\begin{align*}
r_n(z) = \prod_{k=1}^n \frac{(k-1)\lambda(z) + \theta(z) }{k\mu(z) + \vartheta(z) }\,;
\end{align*} 
the rest of the construction is carried as in Example \ref{example-population}.
The  invariant measure $\eta$ can be used to study the size distribution of growth stocks. Such a model captures 
 seasonal   and environmental/external effects that may impact the growth of stock values. 
\end{example}

\section{Diffusive environment} \label{sec-diffusive} 

In this section, we consider birth-death processes with diffusive rates, where the environment process can be a general reflected jump diffusion process. 
We include the case of oblique reflection, and consider piecewise smooth domains. There is a well-developed theory of such processes \cite{stroock1971diffusion,tanaka1979stochastic}; see also an extensive bibliography in \cite{Lithuania}.

\subsection{Reflected  jump diffusion process}
Let us introduce a setting for models with a diffusive environment. 
We consider a jump diffusion process $\wtZ_n (t)$ moving in a piecewise smooth domain $D\subset \RR^d$ 
with smooth drift function $z\in D\mapsto b_n(z)$ and  non-degenerate diffusion matrix function $z\in D\mapsto\sigma_n(z)$,  
depending on $n\in{\mathbb N}$, and with  jumps and reflections described below. \footnote{We do not discuss at this point the exact conditions guaranteeing
the existence and uniqueness of process $\wtZ_n(t)$ in a general setting. In the considered examples, the existence and uniqueness
will be directly verified.}

A {\it domain} in $\RR^d$ is the closure of an open connected subset. A domain is called {\it smooth} if its boundary is a $(d-1)$-dimensional $C^2$ manifold. Consider an intersection of $m$ smooth domains $D_1, \ldots, D_m$: 
$
D = \cap_{i=1}^m D_i,
$
and assume that it has a boundary $\partial D$ with $m$ $(d-1)$-dimensional  {\it faces:} $F_i := \partial D \cap\partial D_i$. Then $D$ is called a {\it piecewise smooth domain} in $\RR^d$.  Denote by ${\tt n}_i(z)$ the inward unit normal vector to 
$\partial D_i$ at $z \in F_i$. An example is a {\it convex polyhedron} with $D_i$ being half-spaces.

Define a continuous function $\gamma_i: F_i\to \RR^d$ satisfying $\gamma_i(z) \cdot {\tt n}_i(z)>0$. Let $\ell_i=\{\ell_i(t): t\ge 0\}$ be continuous nondecreasing processes such that $\ell_i$ can only grow on $F_i$, for $i=1,\dots,m$. Given $n\in{\mathbb N}$ and $z\in D$,  let $\varpi_n(z, \cdot)$ be a finite measure on $D$ such that $\varpi_n(z,\cdot)\Rightarrow 
\varpi(z^0,\cdot)$ as $z\to z^0$ in $D$  (weak continuity). Let $J_n(t)$ be a process that is right continuous piecewise constant, with jump measure 
$\varpi (\,\cdot\,,\,\cdot\,)$  (in the course of process $\wtZ_n$ it will be $\varpi_n(\wtZ_n(t-), \cdot)$). 
The  process $\wtZ_n(t)$ is defined as the solution to the stochastic differential equation
\begin{equation}
d\wtZ_n(t) = b_n(\wtZ_n(t)) dt + \sigma_n(\wtZ_n(t)) d W(t) + dJ_n(t) + \sum_{i=1}^m \gamma_i(\wtZ_n(t)) d \ell_i(t),
\end{equation}
where $W(t)$ is a standard  $d$-dimensional Wiener process adapted to the natural filtration. The generator $\mathcal{A}_n$ 
 of $\wtZ_n$ takes the form
\begin{equation} \label{eqn-gen-A}
\mathcal{A}_ng(z) = b_n(z) \cdot \nabla g(z) + \frac{1}{2} \tr(\sigma_n(z)^{\rm T}\sigma_n(z) \nabla^2g(z)) + 
\int_{D} (g(z')-g(z)) \varpi_n(z, dz'), 
\end{equation}
 and acts on a function $g\in{\mathcal D}_z$ where 
$$
{\mathcal D}_z:= \{g \in C_b^2(D):\ \gamma_i(\tilde{z})\cdot \nabla g(\tilde{z}) = 0, \,\tilde{z} \in F_i,\, i=1,\dots,m\}.
$$

In models where the environment process is a jump diffusion in domain $D$, we set ${\mathcal Z}=D$.
  
\subsection{Joint generator} The joint Markov process $(N, Z)$ on $\NN\times \mathcal{Z}$ has the following generator:
\begin{equation}  \label{eqn-gen-L}
\mathcal{L}f(n,z) = \mathcal{M}_zf(n,z) + \beta_nr_n(z)^{-1}\mathcal{A}_nf(n,z),
\end{equation}
for any function $f$ in the domain of $\mathcal{L}$: 
$$
\mathcal{D}= \{f: \NN\times \mathcal{Z}\to \RR\,| f(n,\cdot) \in \mathcal{D}_z\, \forall n \in \NN\}.
$$
Here, $\beta_n$ is the {\it variability coefficient} for the diffusive environment depending on the state $n$, while $r_n(z)$ is the impact factor from the birth-death process as defined in \eqref{eqn-pi-n-BD}. Also,
\begin{equation} \label{eqn-gen-Mz}
\mathcal{M}_z g(n) = \lambda_n(z) (g(n+1) - g(n) ) + {\bf 1}_{n\neq 0} \mu_{n}(z) (g(n-1) - g(n)),
\end{equation}
for any function $g$ in the domain of $\mathcal{M}_z$ for  each given $z \in \mathcal{Z}$.
Denote by $P^t((n,z), \cdot)$ the transition kernel of $ (N,Z)$ for $(n,z) \in \NN\times \mathcal{Z}$. The joint Markov Process $ (N,Z)$ evolves 
as follows: 
\begin{itemize}
\item If $N(t) = n \in \NN$, then the component $Z(t)$ evolves as a reflected jump diffusion in $\mathcal{Z}$ with  generator 
$\beta_n (r_n(z))^{-1} \mathcal{A}_n$. 

\item If $Z(t) = z \in \mathcal Z$, then the component $N(t)$ jumps from $n$ to $n+1$ with rate $\lambda_n(z)$ and to $n-1$ with rate $\mu_n(z)$ (when $n\neq 0$).  That is, $N(t)$ evolves as a continuous-time Markov chain with the generator 
$\mathcal{M}_z$ in \eqref{eqn-gen-Mz}. 
\end{itemize}

\begin{remark} \label{rem-generator}
Under Assumption \ref{as-r0}, we have $r_n(z)\to 0$ as $n\to \infty$ for each $z$, so that $r_n(z)^{-1}\to \infty$ as $n\to\infty$. It 
may appear that in the joint generator in \eqref{eqn-gen-L}, the second component gets large when $n$ is large, i.e., the environment 
changes the states faster for larger values of $n$.  
However, $\beta_n$ can be relatively small so that $\beta_n r_n(z)^{-1}$ is not large when $n$ is large.  
In general, the generator $\mathcal{A}_n$ also depends on $n$, as will be discussed in Section \ref{sec-diff-gen-n} (see equation 
\eqref{eqn-gen-L-An}). 

For example, consider a Markovian queueing system in a random environment where the service speed is  increasing with the congestion level, while servers tend to break more frequently when the service speed is high.  The repairing rate may be also larger in that setting, so that the environment dynamics tends to evolve faster. However, external factors may prevent this from happening, for instance, the high cost of increasing service and/or repair speed.
Our formulation can be used to cover all these settings. \hfill $\Box$
\end{remark}

\subsection{Main results on the stationary distribution}  For simplicity, we will suppose in this section that the
generator ${\mathcal A}_n$ does not depend on $n$: ${\mathcal A}_n={\mathcal A}$. Consequently, the subscript $n$ 
is omitted from the related notation. Next, we make the following assumption.

\begin{assumption}\label{as-diff}
The reflected jump diffusion with generator $\mathcal{A}$ is positive recurrent, and has a unique invariant measure $\nu$, together with boundary measures  $\nu_{F_i}$, $i=1,\dots,m$. 
That is, there exists a stationary version of the process $\widetilde{Z}^* = \{\widetilde{Z}^*(t): t\ge 0\}$ such that  for all $t\ge 0$ \ $\widetilde{Z}^*(t) \sim \nu$ and for each 
$i=1,\dots,m$ and bounded function $f: F_i \to \RR$, 
$$
\mathbb{E} \int_0^t f(\widetilde{Z}^*(s)) d \ell_i(s) = t \int_{F_i} f(z) \nu_{F_i} (dz)\,.
$$
 Moreover, the measure $\nu$ satisfies
 \begin{equation} \label{eqn-Xi-finite-diff}
 \Xi := \sum_{n=0}^\infty \int_{\mathcal{Z}} r_n(z) \nu(dz) <\infty\,. 
 \end{equation}
\end{assumption} 

\begin{theorem} \label{thm-BD-diff} 
Under Assumptions \ref{as-r0} and \ref{as-diff}, the process $(N, Z)$ has  a finite invariant measure $\eta$ on
${\mathbb N}\times{\mathcal Z}$:
\begin{equation} \label{eqn-eta-diff}
\eta(\{n\},dz) =  r_n(z) \nu(dz).
\end{equation}
The corresponding probability measure is
\begin{equation} \label{eqn-pi-diff}
\pi(\{n\},dz) = \Xi^{-1} r_n(z) \nu(dz).
\end{equation}
The boundary measures $\pi_i$ for $F_i$ are given by 
\begin{equation}
\pi_i(\{n\},dz) = \Xi^{-1} r_n(z) \nu_{F_i}(dz).
\end{equation}
The process $(N,Z)$ is ergodic: for each $(n,z) \in \NN\times \mathcal{Z}$, 
\begin{equation}
\| P^t((n,z), \cdot) - \pi(\cdot)\|_{\mathrm{TV}} \to 0 \qasq t\to\infty.
\end{equation}
\end{theorem}

\begin{remark} \label{rem-diff-gen-n}
Note that the coefficient $\beta_n$ does not appear in the invariant measure $\eta(\{n\},dz)$  in \eqref{eqn-eta-diff}. 
In the  jump environment  with the joint generator \eqref{eqn-R-BD}, we had 
$$R[(n,z), (n,z')] = r_n(z)^{-1}  \tau_n(z,z')$$ 
where $\tau_n(z,z')$ depends on $n$. 
In that setting, the measure $\nu$ in Assumption \ref{as-Tn} is also independent on $n$. 
In the construction of the joint generator $\mathcal{L}$ in \eqref{eqn-gen-L}, the second component 
$ \beta_nr_n(z)^{-1}\mathcal{A}f(n,z)$ is purposely made in the multiplicative form $ \beta_n\mathcal{A}f(n,z)$ such that the dependence 
on $n$ is through the constant $\beta_n$.  However, this multiplicative construction does not entail  any effect of $\beta_n$ upon the invariant 
measure.  In Section \ref{sec-diff-gen-n}, we discuss a more general construction 
where the generator $\mathcal{A}_n$ of the reflected (jump) diffusion process depends on the state of the birth-death process through the reflection domains.  It is interesting to study further more general constructions of $\mathcal{A}_n$ with dependence on $n$. \hfill $\Box$
\end{remark}

\subsection{Examples} \label{sSect3.4} 
An M/M/1 queue with diffusive rates was already considered in \cite[Section 3]{PSBS} and can be regarded as a 
special case of the M/M/$K$ queue below, so it is omitted for brevity. 
 We start with some cases of  M/M/$\infty$ queues with diffusive rates  in Examples \ref{exm:RBM-arrival}
--\ref{example-MMinfty-compact}. 
 For simplicity, in these examples the drift and the diffusion coefficients are taken to be constant,
although an extension of the argument to the case where these coefficients varying with $n$ and $z$ is straightforward. 
Also the jump process $J(t)$ is disregarded, as well as the reflection processes $\ell_i (t)$. 
We also repeatedly use the notation $\wtZ (t)$
as an alternative to $Z(t)$,  
to describe an environmental process {\sl per se}, for a fixed value $n\in{\mathbb N}$.

Our goal is to construct a finite invariant $\eta$ for the joint Markov process $(N,Z)$. The invariant probability distribution $\pi$
will be obtained after the normalization.

\begin{example} \label{exm:RBM-arrival}
In this example, the arrival rate  $\lambda(\cdot)$ is  a reflected Brownian motion (RBM)  in the positive half-line $\RR_+$ (i.e., ${\mathcal Z}=D
=(0,\infty)$), with a negative drift $-c$ and diffusion coefficient $\sigma$: $\lambda(\cdot) = \widetilde{Z}$ where 
$ \widetilde{Z}(t) = -ct + \sigma W(t) +\ell(t) $ with $c,\sigma >0$ and $\ell(t)$ is the regulating process (continuous and nondecreasing, $\ell(0)=0$ and $\ell(t)$ only increases at times  when $\wtZ(t)=0$).  Also, let the service rate $\mu(t) \equiv \mu$ be a constant. 
The process $\widetilde{Z}$ has the exponential stationary distribution
\begin{equation}
\label{eq:exponential}
\nu(dz) = \frac{2c}{\sigma^2} \exp\Big(- \frac{2c}{\sigma^2}z \Big)dz,\quad  z>0.
\end{equation}
We also have $r_n(z) = (z/\mu)^n/n!$ where $n\in{\mathbb N}$.  Condition \eqref{eqn-Xi-finite-diff} means that 
\begin{align*}
\Xi  &=  \sum_{n=0}^\infty \int_0^\infty r_n(z) \nu(dz) = \int_0^\infty \Big( \sum_{n=0}^\infty\frac{(z/\mu)^n}{n!} \Big) \frac{2c}{\sigma^2} \exp\Big(- \frac{2c}{\sigma^2}z \Big)dz \\
& =\int_0^\infty e^{z/\mu}  \frac{2c}{\sigma^2} \exp\Big(- \frac{2c}{\sigma^2}z \Big)dz  =  \frac{2c}{\sigma^2}\int_0^\infty  e^{ - \big( \frac{2c}{\sigma^2}-1/\mu\big) z}dz  \\
& =  \frac{2c}{\sigma^2}\Big( \frac{2c}{\sigma^2}-1/\mu\Big)^{-1}<\infty,
\end{align*}
which is the case when 
\begin{equation}
\label{eq:Xi-finite1}
\frac{2c}{\sigma^2}-\frac{1}{\mu} >0.
\end{equation}
Under condition~\eqref{eq:Xi-finite1},  a finite invariant measure of the joint process $(N,Z)$  on 
${\mathbb N}\times{\mathbb R}_+$ is given by: 
$$
\eta(\{n\},dz) = r_n(z)\nu(dz)  = \frac{(z/\mu)^n}{n!}  \frac{2c}{\sigma^2} \exp\Big(- \frac{2c}{\sigma^2}z \Big)dz\,. 
$$
\end{example}

\begin{example}\label{exm:RBM-service} 
 Let $\lambda(\cdot) \equiv \lambda$ be a constant and $\mu(t)$ be an RBM with a negative drift in  $(\mu_0, \infty)$ 
with $\mu_0>0$. We have the same invariant measure $\nu(dz)$ from~\eqref{eq:exponential}, shifted by $\mu_0$, with the 
 density
\begin{equation}\label{eqn-nu-MMInf-service}
\nu(dz) =\frac{2c}{\sigma^2} \exp\Big(- \frac{2c}{\sigma^2}(z - \mu_0)\Big)d z,\quad z>\mu_0.
\end{equation}
Next, $r_n(z) = (\lambda/z)^n/n!$.  The condition  \eqref{eqn-Xi-finite-diff} is satisfied: 
\begin{align*}
\Xi & =  \sum_{n=0}^\infty \int_{\mu_0}^\infty r_n(z) \nu(dz) = \int_{\mu_0}^\infty 
\Big( \sum_{n=0}^\infty  \frac{(\lambda/z)^n}{n!} \Big) \frac{2c}{\sigma^2} 
\exp\Big(- \frac{2c}{\sigma^2}(z - \mu_0) \Big)dz \\ &= \int_{\mu_0}^\infty 
e^{\lambda/z}  \frac{2c}{\sigma^2} \exp\Big(- \frac{2c}{\sigma^2}(z - \mu_0) \Big)dz < \infty.
\end{align*}
(If we set $\mu_0 = 0$, then this integral would be infinite.) A finite invariant measure of the joint process $(N,Z)$ is 
$$
\eta(\{n\},dz) = r_n(z)\nu(dz)  = \frac{(\lambda/z)^n}{n!}  \frac{2c}{\sigma^2} \exp\Big(- \frac{2c}{\sigma^2}(z-\mu_0) \Big)dz\,. 
$$
\end{example}

\begin{example}  \label{example-MMinfty-2dRBM}  Here we assume that the pair $(\lambda, \mu)$ forms a two-dimensional RBM 
in the shifted positive quadrant $\RR_+ \times (\mu_0, \infty)$ (the restriction on $\mu_0>0$ emerges in \eqref{eq:Xi-finite2} below). 
Specifically, $\lambda(t)=\wtZ_1$ and $\mu(t)=\wtZ_2(t)$, $t\geq 0$, where $\wtZ = (\wtZ_1,\wtZ_2)$ is an RBM in 
${\mathcal Z}=\RR_+ \times (\mu_0, \infty)$, given by 
$\wtZ(t) = c t + \sigma W(t) + R Y(t)$, where $c = (c_1,c_2)^{\rm T}\in{\mathbb R}^2$ is a drift vector, $\sigma$ is a $2\times 2$ 
positive-definite diffusion matrix, $R$ is a (non-singular) reflection matrix, 
and $Y(t)=(Y_1(t),Y_2(t))^{\rm T}$ is a regulating process (continuous and 
nondecreasing, with $Y(0)=0$ and $Y_i$ increasing at times when $\wtZ_i(t)=0$ for $i=1,2$). Suppose the process
 $\wtZ$ is positive recurrent, and the skew-symmetry condition is satisfied: $2 \Sigma = R D + D R^{\rm T}$ where 
$\Sigma = \sigma^2$ and $D =\text{diag}\{\Sigma_{ii},\,i=1,2\}$. 
It is shown in \cite{harrison1981distribution} that the invariant measure $\nu$ for $\wtZ$ has an explicit product form:
$$\nu(dz_1,d z_2) =  \alpha_1\alpha_2 e^{-\alpha_1 z_1 - \alpha_2(z_2-\mu_0)}d z_1 d z_2,\quad z_1>0,\;z_2>\mu_0,$$  
where $\alpha_i = 2 c_i \xi_i/\Sigma_{ii}$ and $\xi= R^{-1} c$. See also the survey \cite{williams1995semimartingale}. 
Assume that 
\begin{equation}\label{eq:Xi-finite2}\mu_0>\dfrac{1}{\alpha_1}\,.\end{equation}

We have $r_n(z) = (z_1/z_2)^n/n!$. We also obtain
\begin{align*}
\Xi &= \sum_{n=0}^\infty \int_0^\infty\int_{\mu_0}^\infty  \frac1{n!}\left(\frac{z_1}{z_2}\right)^n\alpha_1\alpha_2 e^{-\alpha_1z_1 - \alpha_2(z_2-\mu_0)} dz_1 dz_2 \\ & = \int_0^{\infty}\int_{\mu_0}^{\infty}e^{z_1/z_2}\alpha_1\alpha_2e^{-\alpha_1z_1 - \alpha_2(z_2-\mu_0)}dz_1dz_2 \\
& = \alpha_1\alpha_2 \int_{\mu_0}^{\infty}  \Big(  \int_0^{\infty}
e^{-(\alpha_1 - 1/z_2) z_1}dz_1 \Big)  e^{-\alpha_2 (z_2-\mu_0)}dz_2\\
& =   \alpha_1\alpha_2  \int_{\mu_0}^{\infty}
\frac{1}{\alpha_1- 1/z_2}  e^{-\alpha_2 (z_2-\mu_0)}dz_2\\
& <   \alpha_1\alpha_2  \int_{\mu_0}^{\infty}
\frac{1}{\alpha_1- 1/\mu_0}  e^{-\alpha_2(z_2-\mu_0)}dz_2\\
& = \frac{ \alpha_1\alpha_2}{\alpha_1- 1/\mu_0} \times \frac{1}{\alpha_2} =  \frac{ \alpha_1  }{\alpha_1- 1/\mu_0}<\infty. 
\end{align*}
Here in the fourth and fifth lines we have used the condition \eqref{eq:Xi-finite2}. 
Thus, a finite invariant measure of the joint process $(N,Z)$ on ${\mathbb N}\times{\mathcal Z}$ is given by 
$$
\eta(\{n\},dz_1, dz_2) = r_n(z)\nu(dz)  =\frac{(z_1/z_2)^n}{n!}    \alpha_1\alpha_2 e^{-\alpha_1 z_1 - \alpha_2(z_2-\mu_0)}d z_1 d z_2\,.  
$$

\end{example}

\begin{example} \label{ex3.4}
Let us modify Example~\ref{exm:RBM-arrival} to make the arrival rate a reflected Ornstein-Uhlenbeck process: $\lambda(t) =\wtZ(t)$, 
where $\wtZ$ solves the stochastic differential equation  with reflection on $\RR_+$:
\begin{equation}
\label{eq:OU}
d\wtZ(t) = -c\wtZ(t)\,dt + \sigma dW(t) + d\ell(t),
\end{equation}
where $c,\sigma >0$ and $\ell(t)$ is the regulating process. By \cite{ward2003properties}, the  invariant measure for $\wtZ $ is one-sided Gaussian:
\begin{equation}
\label{eq:OU-measure}
\nu (dz) = \exp\left(-\frac{2c}{\sigma^2}z^2\right)dz,\quad  z>0.
\end{equation}
Similarly to Example~\ref{exm:RBM-arrival}, $r_n(z) = (z/\mu)^n/n!$. 
Condition \eqref{eqn-Xi-finite-diff} is satisfied for any $\mu >0$ since 
\begin{align*}
\int_0^{\infty}&\sum_{n=0}^{\infty}r_n(z)v(z)\,dz = \int_0^{\infty}\exp\left(\frac{z}{\mu} - \frac{2c}{\sigma^2}z^2\right)\,dz  \\
&=  \exp\Big(\frac{\sigma^2}{2c\mu^2}\Big) \int_0^\infty \exp\left( -\frac{2c}{\sigma^2} \Big(z-  \frac{\sigma^2}{4c\mu}\Big)^2   \right) d z =  \sigma \sqrt{\frac{\pi}{2c}} \Phi\Big(- \frac{\sigma^2}{4c\mu}\Big)<\infty, 
\end{align*} 
where $\Phi(\cdot)$ is the  cumulative distribution function of the standard normal distribution. 
Thus,  a finite invariant measure of the joint process $(N,Z)$ is given by 
$$
\eta(\{n\},dz) = r_n(z)\nu(dz)  = \frac{(z/\mu)^n}{n!} \exp\left(-\frac{2c}{\sigma^2}z^2\right)dz.
$$
\end{example}

\begin{example}  \label{ex3.5}
Next, let us modify Example~\ref{exm:RBM-service} to make the service rate a reflected Ornstein-Uhlenbeck process~\eqref{eq:OU}, but on $[\mu_0, \infty)$ for some $\mu_0>0$.
  Its stationary distribution has density proportional to $\nu$ from~\eqref{eq:OU-measure}. 
Here we have $r_n(z) = (\lambda/z)^n/n!$.
So  condition \eqref{eqn-Xi-finite-diff}
is satisfied for any $\lambda\geq 0$ by verifying  that
$$
\int_{\mu_0}^{\infty}\sum_{n=0}^{\infty}r_n(z)v(z)\,dz = \int_{\mu_0}^{\infty}\exp\left[\frac{\lambda}z - \frac{2c}{\sigma^2}(z - \mu_0)^2\right]\, dz < \infty.
$$ 
(The integral is finite since $z>\mu_0>0$.)  Then the invariant measure of the joint process $(N,Z)$ is given accordingly. 

One could also consider the model  where the arrival and service form a two-dimensional reflected  Ornstein-Uhlenbeck process 
in the positive orthant ${\mathbb R}_+\times (\mu_0,\infty )$. However, the explicit expression for its invariant measure is not known. 

Examples \ref{ex3.4} and \ref{ex3.5} can be extended to one-dimensional reflected diffusions with a piecewise linear drift, which will 
have truncated Gaussian invariant measures, see \cite{browne1995piecewise}. 
\end{example}

\def\ov{\overline}

\begin{example}\label{example-MMinfty-compact}
Examples \ref{exm:RBM-arrival}--\ref{ex3.5} can be modified to include constrained diffusions. 
%In the previous examples of $M/M/\infty$ queues with diffusive rates, we have considered the cases of the parameters being reflected diffusions with drifts. One can consider the parameters being diffusions constrained in compact sets. For example, we can take a constant arrival rate and a service rate being a RBM without drift on $[\mu_0, \mu_1] \subset (0, \infty)$. Alternatively, we can take a constant service rate and arrival rate a RBM without drift on $[\lambda_0, \lambda_1] \subset (0,\infty)$. Yet another alternative: we can take the arrival and service rates as a two-dimensional RBM on a piecewise smooth region which is a compact subset of $\RR_+^2$. In all these cases, the stationary distribution is uniform in one or two dimensions, the ratio of the arrival and service rates $\lambda(z)/\mu(z)$ is bounded, and thus the condition $\Xi < \infty$ is satisfied, and the invariant measures are then explicitly given accordingly. 
For instance, one can 
think of an M/M/$\infty$ queue where the diffusive pair $(\lambda ,\mu)=\wtZ$ moves within a compact domain 
$D\subset{\mathbb R}_+\times{\mathbb R}_+$ with a piecewise smooth boundary $\partial D$ and with the normal reflection at $\partial D$. 
% such that the closure $\ov D=D\cup\partial D$ is separated from the half-line $\{(\lambda ,\mu):\;\mu =0,\lambda >0\}$, with the normal reflection at $\partial D$. 
For general `nice' 
drift and diffusion coefficients, the Lebesgue 
measure on $D$ will be invariant for the process $\wtZ$, and the above-type calculations  could be carried through, guaranteeing a finite
invariant measure for the joint process $(N,Z)$. 
\end{example}

\begin{example}
\label{exm:RBM-arrival-jump}
Let us modify Example \ref{exm:RBM-arrival} to include Poisson jumps. Consider a  jump diffusion process $\wtZ$
on half-line ${\mathbb R}_+$ with  constant negative drift $-c$ and diffusion 
 coefficient $\sigma >0$, and with i.i.d. jumps with  intensity  $\kappa$ and 
a distribution of jump size ${\rm K}  (dz)$ (so that $\kappa{\rm K}(\cdot)$ 
is the spectral measure), supported on $\mathbb R_+$ (so that the jumps are to the right).  The generator of 
$\wtZ$ is given by 
\begin{equation}
\label{eq:generator-jumps}
\mathcal Lf(x) = -cf'(x) + \frac{\sigma^2}2f''(x) + \kappa\int_0^{\infty}[f(x+y) - f(x)]\,{\rm K}(\mathrm{d}y),\quad x>0,
\end{equation}
for $f : \mathbb R_+ \to \mathbb R$ in $C^2$ with condition $f'(0) = 0$. The combined drift in $\wtZ$ equals
$-c + \kappa\overline{\rm K}$ where $\ov{\rm K}:=\int_0^\infty y{\rm K} (dy)$. If this drift is negative (assuming
$c>\kappa{\overline{\rm K}}$), then the  process 
$\wtZ$ is ergodic, see \cite[Section 6]{sarantsev2016exp}.  Let $\nu$ denote the stationary distribution of $\wtZ$.
 The MGF $\Psi_{\nu}(u)=\int_0^\infty e^{uy}\nu (dy)$ can be determined by 
rewriting~\eqref{eq:generator-jumps} in terms of the 
Laplace transform.  Denote by $\Psi_{\rm K}$ the MGF of the jump measure $\rm K$:  
$\Psi_{\rm K}(u)=\int_0^\infty e^{uy}{\rm K}(dy)$. After adjusting 
the notation of \cite[Example 4.3]{forward}, the MGF  $\Psi_{\nu}$  becomes
\begin{equation}
\label{eq:MGF-levy}
\Psi_{\nu}(u) := \frac{Mu}{F(u)}\,, \;\hbox{ where }\; F(u) := cu - \frac{1}{2}\sigma^2u^2 - 
\kappa\Psi_{\rm K}(u) + \kappa\,.\end{equation}
Note that $F$ is a concave function, and it has two zeros: $u = 0$ and $u = u_0 > 0$. The latter is true since $F'(0) 
= c - \kappa\Psi_{\rm K}'(0) = c - \kappa\overline{\rm K}> 0$. 
Similarly to Example~\ref{exm:RBM-arrival}, the quantity $\Xi$ is equal to the value of the MGF of this stationary distribution: 
$\Xi = \Psi_{\nu}(\mu^{-1})$.  Therefore, $\Xi <\infty$ if $\mu^{-1} < u_0$. This condition is an analogue of~\eqref{eq:Xi-finite1}.

Similar extensions of continuous diffusive cases to jump-diffusion cases could be done for the previous examples.
\end{example}

\begin{example} \label{example-MMK-diff}
In this example, we consider M/M/$K$ queues with the following types of diffusive arrival and/or service rates:

(a)  $\lambda$ is an RBM  in $[0, \mu K]$ where $\mu = \mathrm{const}$;

(b)  $\lambda = \mathrm{const}$ and $\mu$ is an RBM with a negative drift in $[\mu_0, \infty)$ where $\mu_0>\lambda/K$;

(c)  $(\lambda, \mu)$ is a two-dimensional RBM in a wedge $\{z=(z_1,z_2)\in \RR^2_+: z_1 \le z_2 K\}$. 

In all three cases, we have $r_n(z)$ given by \eqref{eqn-rn-MMK} in Example \ref{example-MMK-jump}. Case (a) is similar to
 the one discussed in  Example \ref{exm:RBM-arrival-jump}. 
 In 
Case (b), we have an invariant measure $\nu(dz)$ given in \eqref{eqn-nu-MMInf-service}, so 
 condition \eqref{eqn-Xi-finite-diff} requires that 
\begin{align*}
\Xi 
&= \sum_{n}\int_{\mu_0}^\infty r_n(z) \nu(dz) \\ & 
 = \int_{\mu_0}^\infty  \left(\sum_{n=0}^{K-1}\frac{(\lambda/z)^{n}}{n!} +  \sum_{n=K}^\infty \frac{(\lambda/z)^n }{K!K^{n-K}} \right)  \frac{2c}{\sigma^2} \exp\Big(- \frac{2c}{\sigma^2}(z - \mu_0) \Big)dz <\infty.
\end{align*}
Case (c) seems more challenging and requires a detailed analysis  of the invariant measure of two-dimensional RBMs in a wedge (see, e.g.,  \cite{dieker2009reflected}).

For an M/M/$K$/$0$ queue, there is no stability concern, and the queueing state  space is finite. Hence, the underlying diffusive environment for either arrival or service rates or both can be any reflected jump diffusion of the above type as long as the invariant measure $\nu$ exists. Then a finite invariant measure for  $(N,Z)$ is $\eta(\{n\}, dz) = r_n(z) \nu(dz)$ where $r_n(z)$ is given in Example \ref{example-MMK0}. 
\end{example}

\begin{example} \label{example-MMK+M-diff}
For  an M/M/$K$+M queue, as discussed in Example \ref{example-MMK+M}, the arrival, service and abandonment rates 
 $\lambda$, $\mu$ and $\gamma$ can all depend on the environment. An interesting case is where the triple 
$(\lambda(\cdot), \mu(\cdot), \gamma(\cdot))$ evolves as an RBM $\wtZ=(\wtZ_1,\wtZ_2,\wtZ_3)$ in a shifted octant 
$\RR_+\times (\mu_0,\infty)\times (\gamma_0,\infty) \subset \RR_+^3$, of the form $\wtZ(t)=ct +\sigma W(t) + RY(t)$,
similar to Example \ref{example-MMinfty-2dRBM} (but in three dimensions). As shown in  \cite{harrison1981distribution}, 
under the positive recurrence and skew-symmetry conditions  the process $\wtZ$ has a product-form invariant measure 
$$\nu(dz_1,d z_2, d z_3) =  \alpha_1\alpha_2 \alpha_3 e^{-\alpha_1 z_1 - \alpha_2(z_2-\mu_0) - \alpha_3 (z_3-\gamma_0)}d z_1 d z_2 d z_3\,.$$ 
Let us assume that constants $\mu_0>0$ and $\gamma_0>0$ satisfy 
$\gamma_0>1/\alpha_1$ where $\alpha_1=2c_1\xi_1/\Sigma_{11}$. Cf. Example \ref{example-MMinfty-2dRBM}. 
 
We have the same formula  for $r_n(z)$ as in \eqref{eqn-rn-MMK+M}, with $\rho(z)=z_1/z_2$ and $\beta(z) = z_1/z_3$. 
Thus, condition \eqref{eqn-Xi-finite-diff} requires that 
\begin{align*}
\Xi  
&= \sum_{n}\!\int\limits_{\RR_+\times [\mu_0,\infty)\times [\gamma_0,\infty)}\!\!\!\! r_n(z) v(dz)  \\
&=\!\!\! \int\limits_{\RR_+\times [\mu_0,\infty)\times [\gamma_0,\infty)}\!\!\! \left(\sum_{n=0}^{K-1} \frac{(z_1/z_2)^n}{n!} + \frac{(z_1/z_2)^K}{K!} e^{z_1/z_3} \right)    \\
& \qquad \qquad  \times \alpha_1\alpha_2 \alpha_3 e^{-\alpha_1 z_1 - \alpha_2(z_2-\mu_0) - \alpha_3 (z_3-\gamma_0)}d z_1 d z_2 d z_3 < \infty.
\end{align*}
It is easy to check, similarly to Example \ref{example-MMK+M}, that the first component of the integral is finite as 
$z_2>\mu_0>0$ while the second and third components are finite for  $z_2>\mu_0>0$ and $z_3>\gamma_0>1/\alpha_1$. 
Thus, a finite invariant measure for $(N,Z)$ is 
\begin{align*}
&\eta(\{n\}, dz_1,dz_2, dz_3)= \left(\sum_{n=0}^{K-1} \frac{(z_1/z_2)^n}{n!} + \frac{(z_1/z_2)^K}{K!} e^{z_1/z_3} \right)  \\
&\qquad \qquad\qquad\qquad \qquad \qquad  \times  \alpha_1\alpha_2 \alpha_3 e^{-\alpha_1 z_1 - \alpha_2(z_2-\mu_0) - \alpha_3 (z_3-\gamma_0)}
d z_1 d z_2 d z_3\,. \end{align*}
The cases where only one or two of the three parameters are diffusive can be considered in a similar manner.  
\end{example}

\begin{example} \label{example-population-diff}
For the linear population growth model with immigration  from Example \ref{example-population}, the parameters 
$\lambda, \mu, \theta$ can  again be all diffusive, represented by a three-dimensional RBM $\wtZ$ in 
$\RR_+\times [\mu_0,\infty)\times \RR_+$, as was discussed in Example \ref{example-MMK+M-diff}. 
The constant $\mu_0$ must satisfy $\mu_0>1/\alpha_1$. The process $\wtZ$ has a 
product-form invariant measure
 $\nu(dz_1,d z_2, d z_3) =  \alpha_1\alpha_2 \alpha_3 e^{-\alpha_1 z_1 - \alpha_2(z_2-\mu_0) - \alpha_3 z_3}d z_1 d z_2 d z_3$. 
 Here we have the  same formula for $r_n(z)$ as in \eqref{eqn-rn-population} with $\rho(z) = z_1/z_2$. 
 Condition \eqref{eqn-Xi-finite-diff} 
 requires now that 
\begin{align*}
&\Xi = \!\!\!\int\limits_{\RR_+\times [\mu_0,\infty)\times \RR_+}\!\!\!\sum_n
\frac{(z_1/z_2)^n}{n!} \prod_{k=1}^n \Big(k-1 + \frac{z_3}{z_1}\Big)\\
& {}\qquad\qquad\times \alpha_1\alpha_2 \alpha_3 e^{-\alpha_1 z_1 - \alpha_2(z_2-\mu_0) 
- \alpha_3 z_3}d z_1 d z_2 d z_3 <\infty 
\end{align*}
which  is finite under the condition $z_2>\mu_0>1/\alpha_1$, similarly  to Example \ref{example-MMK+M-diff}. 
Thus  a finite invariant measure of $(N,Z)$ is 
\begin{align*}
&\eta(\{n\}, dz_1,dz_2, dz_3) = \frac{(z_1/z_2)^n}{n!} \prod_{k=1}^n \Big(k-1 + \frac{z_3}{z_1}\Big)\\
&\qquad\qquad\qquad \qquad \qquad \times   \alpha_1\alpha_2 \alpha_3 e^{-\alpha_1 z_1 - \alpha_2(z_2-\mu_0) - \alpha_3 z_3}d z_1 d z_2 d z_3\,. 
\end{align*}

The growth stock model in Example \ref{example-growthstock} with four parameters can also be considered analogously, with 
the parameters evolving as a four-dimensional RBM in a subset of $\RR_+^4$. 

\end{example}

\subsection{ Reflected jump diffusion environment with a variable domain.} \label{sec-diff-gen-n}
In this section we consider a more general setup where the domain of the (jump) diffusion varies with the state $n$ 
of the birth-death process. Such a setup was considered in Section 3.4 of \cite{PSBS}; here we review the construction 
and provide an example.

For each $n\in \NN$, let $D_n \subset \mathbb{R}^d$ be a piecewise smooth domain with $m_n$ faces 
$F^{(n)}_1, \dots, F^{(n)}_{m_n}$ of the boundary $\partial D_n$ and reflection vector fields $f^{(n)}_i: F^{(n)}_i \to \RR^d$. 
 Set $\mathcal{Z}=\cup_{n\in\NN} D_n$. We construct the joint Markov  process $(N, Z)$ on $\NN\times \mathcal{Z}$ 
via the following generator  ${\mathcal L}$ and its domain
${\mathcal D}$: 
\begin{align}
\label{eqn-gen-L-An}
\begin{split}
\mathcal{L}f(n,z) &= \mathcal{M}_zf(n,z) + \beta_nr_n(z)^{-1}\mathcal{A}_nf(n,z),
\\ \mathcal{D} &= \{f: \NN\times \mathcal{Z}\to \RR\,:\, f(n,\cdot) \in{\mathcal D}^n_z\, \quad  \forall\; n \in \NN\}.
\end{split}
\end{align}
 The generator $\mathcal A_n$ and its domain ${\mathcal D}^n_z$ are given by
\begin{align}\label{eqn-gen-An} \begin{split}
\mathcal{A}_n g(z) &= b_n(z) \cdot \nabla g(z) + \frac{1}{2} \tr(\sigma_n(z)^{\rm T}\sigma(z) \nabla^2g(z)) 
+ \int_{D_n} (g(z')-g(z)) \varpi_n(z, dz'), \\
{\mathcal D}^n_z &= \{g \in C_b^2(D_n)\ :\ \gamma_i(\tilde{z})\cdot \nabla g(\tilde{z}) = 0, \,\tilde{z} \in F^{(n)}_i,\, 
\end{split}\end{align}

We modify the conditions in Assumption \ref{as-diff} as follows. 
Here we assume that  
the reflected jump diffusion with generator $\mathcal{A}_n$ is positive recurrent and has a finite invariant measure 
$\nu_n$,  with boundary measures $\upsilon^{(n)}_{F^{(n)}_i}$, $i=1,\dots,m_n$. 
Moreover, the measures $\nu_n$ satisfy the property similar to \eqref{eqn-Xi-finite-diff}
 \begin{equation}\label{eqn-Xi-finite-diff2}
 \Xi := \sum_{n=0}^\infty \int_{D_n} r_n(z) \nu_n(dz) <\infty. 
 \end{equation}

 Then, by modifying the proof of Theorem \ref{thm-BD-diff}, we can show that under the above conditions, the joint 
Markov process $(N,Z)$ has a unique invariant  probability distribution 
 $$
 \pi(\{n\}, dz) = \Xi^{-1} r_n(z) \nu_n (dz), 
 $$
 and the corresponding boundary measures   $\pi^{(n)}_{F^{(n)}_i}$ on $F^{(n)}_i$ (if there is a reflection) 
are given by 
 $$
  \pi_{F^{(n)}_i}(\{n\}, dz) = \Xi^{-1} r_n(z) \upsilon^{(n)}_{F^{(n)}_i} (dz), \quad i =1,\dots, m_{n}. 
 $$

\begin{example}
Recall that in  Examples \ref{exm:RBM-arrival}--\ref{example-MMinfty-compact}, we have discussed the $M/M/\infty$ queues 
with arrival and/or service rates being a reflected diffusion.
However, the domain may depend on the state of the queue. For instance, Example 
\ref{exm:RBM-service}, one can consider the service rate being an RBM with a negative drift in 
$[\mu_n, \infty)$ where $\mu_n>0$.  We then  obtain an
invariant measure 
\begin{equation*}
\nu_n(dz) =\frac{2c}{\sigma^2} \exp\Big(- \frac{2c}{\sigma^2}(z - \mu_n)\Big)d z\,,\quad z>\mu_n \,.
\end{equation*}
Condition \eqref{eqn-Xi-finite-diff2} takes the form 
\begin{align*}
\Xi
 & =  \sum_{n=0}^\infty \int_{\mu_n}^\infty r_n(z) \nu_n(dz) 
= \sum_{n=0}^\infty \int_{\mu_n}^\infty \Big(   \frac{(\lambda/z)^n}{n!} \Big) \frac{2c}{\sigma^2} \exp\Big(- \frac{2c}{\sigma^2}(z - \mu_n) \Big)dz 
 < \infty\,.
\end{align*}
A finite invariant measure for $(N,Z)$ reads 
$$
\eta(\{n\},dz) = r_n(z)\nu_n(dz)  = \frac{(\lambda/z)^n}{n!}  \frac{2c}{\sigma^2} 
\exp\Big(- \frac{2c}{\sigma^2}(z-\mu_n) \Big)dz\,. $$
 Similar extensions can be done in the other examples in  Section \ref{sSect3.4}. 
\end{example}

\section{Exponential Convergence to Stationarity} \label{sec-conv-rate}

 In this section we study the rate of convergence to stationarity of the joint Markov process $(N,Z)$ constructed in the previous two sections. 
We focus on the diffusive random environment; the jump environment can be studied similarly. We  consider convergence in the total variation distance, cf. \eqref{eq:TVD}. 
 If $P^t(x, \cdot)$ is the transition function 
of a process, and $\pi$ is its stationary distribution, we say that $r(\cdot)$ is the convergence rate if for all $x$ and $t > 0$ 
we get:
$$
\|P^t(x, \cdot) - \pi(\cdot)\|_{\mathrm{TV}} \le C(x)r(t)
$$
for some $C(x)$. We consider two scenarios where the convergence rate is exponential: $r(t) = e^{-ct}$ with a constant $c > 0$. In each scenario, we assume an exponential convergence rate for the process of random environment. Likewise, in each scenario, we impose conditions on the birth and death rates which result in an exponential rate for the process $(N,Z)$. In the first scenario,   the underlying birth-death process satisfies a stability condition  \eqref{eq:main-bounds} in Assumption \ref{asmp:bounds}, which covers queueing models such as M/M/1 and M/M/$K$ queues. On the other hand, in the models like M/M/$\infty$, M/M/$K$/$0$ and M/M/$K$+M, the underlying birth-death process does not require any condition for stability. In those models only mild assumptions upon the birth and death rates are needed, summarized in Assumption \ref{asmp:bounds-s2} as we discuss in the second scenario. 

 Classic articles on exponential convergence for a general Markov process are 
\cite{meyn1993stabilityII, meyn1993stabilityIII}. They use Lyapunov functions $V$ for which ${\mathcal G}V \le - kV$ outside a `small' set  in the state space of the process. Here
$\mathcal G$ is the generator of the Markov process, and $k > 0$ is a constant. There exists a substantial literature on this topic.  

An important property of continuous-time Markov processes on the real line $\mathbb R$ is {\it stochastic ordering:} two copies $X_1$ and $X_2$ of a process starting from points $x_1 \le x_2$ can be coupled so that $X_1(t) \le X_2(t)$ for all $t \ge 0$ a.s. Many 
common processes satisfy this property, including birth-death processes and reflected diffusions on the half-line. The second author of this article combined the approach from \cite{meyn1993stabilityII, meyn1993stabilityIII} with stochastic ordering and a coupling argument to estimate explicit rates of exponential convergence; cf \cite{sarantsev2016exp}.

In this article, we use a coupling argument to establish the rates of convergence by adapting the  methods from \cite{lindvall1979note} and also generalizing the approach used in \cite{PSBS} for M/M/1 queues in a random environment.

\subsection{The first scenario}

Define 
\begin{equation}\label{eq:pbarz}
q_i(z) := \lambda_i(z) + \mu_i(z)\,,\quad p_i(z) = \frac{\lambda_i(z)}{q_i(z)}\,,\quad i \in \NN,\quad z\in \mathcal{Z}.
\end{equation}

\begin{assumption}\label{asmp:bounds}
The parameters $p_i(z)$ and $q_i(z)$ of the birth-death process satisfy the following bounds: 
\begin{equation}
\label{eq:main-bounds>0}
\overline{q} := \inf\limits_{z \in \mathcal Z} \inf_i q_i(z) > 0\,,
\end{equation}
\begin{equation}
\label{eq:main-bounds}
\overline{p} := \sup_{z \in \mathcal Z}\sup_i p_i(z) < 1/2\,.
\end{equation}
\end{assumption}

\begin{remark}
 Condition \eqref{eq:main-bounds} in Assumption \ref{asmp:bounds} was imposed to get an exponential rate of 
convergence of a birth-death process (without environment states) in \cite[Proposition 2]{lindvall1979note}. For 
queueing examples with a finite number of servers, the second condition $\overline{p}<1/2$ in~\eqref{eq:main-bounds} implies that the 
traffic intensity is less than one. In the linear population growth model in Example \ref{example-population-diff}, a sufficient condition for this property to hold is  that $\lambda(z)<\mu(z)$ for each $z$. 

In addition, condition \eqref{eq:main-bounds} is used in the coupling  argument where a  dominating embedded Markov chain 
with  a parameter $\overline{p}$ is introduced. This is critical in Step 4 of the proof of Theorem \ref{thm-RC-exp} below. See 
also Remark \ref{rem-hittingtime}. \hfill $\Box$
\end{remark}

We now  introduce an assumption of an exponential rate of convergence  for the environmental process with generator 
$\mathcal{A}$ (which is the same as \cite[Assumption 4.1]{PSBS}). It imposes  an exponential-tail condition on the coupling time 
for the environmental process. Examples of processes satisfying this assumption are also given in 
\cite[Section 4.3.2]{PSBS}, for instance,  a RBM on $[0,a]$ in \cite[Chapter 2, Problem 8.2]{karatzas1991stochastic}. 

\begin{assumption}\label{asmp:basic} 
There exist constants $\alpha > 1$ and $\gamma > 0$ such that for all $z_1, z_2 \in \mathcal Z$  the processes 
$Z_1$ and $ Z_2$ with generator $\mathcal A$, starting from $Z_1(0) = z_1$ and $Z_2(0) = z_2$, can be coupled in time $\tau_{z_1, z_2} := 
\inf\{t \ge 0:\, Z_1(t) = Z_2(t)\}$, with 
\begin{equation}\label{eq:gamma}\mathbb P(\tau_{z_1, z_2} \ge t) \le \alpha e^{-\gamma t}\,.
\end{equation}
\end{assumption}

Next, we define the following auxiliary function: 
\begin{equation}
\label{eq:theta}
\theta(\alpha, \beta, \gamma, a) :=   \frac{\beta}{\beta - a}- \frac{a\gamma}{(\beta-a)(\beta+\gamma-a)} \alpha^{-( \beta-a)/\gamma}\, ,
\end{equation}
for any $\alpha>1$, $\beta, \gamma>0$ and $a \ge 0$. This quantity appears to be an upper bound for the MGF of the minimum of an 
exponential random variable and an independent variable with exponential tail (see Lemma \ref{lemma:tech} in Appendix B and   
\cite[Lemma 6.1]{PSBS}).  Also,  define 
\begin{equation}
\label{eq:g}
\begin{split}
&g(s) := \frac{1 - \sqrt{1 - bs^2}}{2\overline{p}s}\,,\quad \text{with}\quad b := 4\overline{p}(1 - \overline{p}),\\
& G(u) := g\left(\frac{\overline{q}}{\overline{q} - u}\right)\,. 
\end{split}
\end{equation}
\begin{remark} \label{rem-hittingtime}
 Following  \cite[Section 14.5]{Feller},  consider a discrete-time random walk $S = (S_n)_{n \ge 0}$ on  the set of integers $\mathbb Z$ taking 
steps $+1$ and $-1$ with probabilities $\overline{p}$ and $1 - \overline{p}$. By  Assumption~\ref{asmp:bounds}, the overall direction of 
this random walk is downward since $\overline{p}<1/2$, that is,  $\mathbb E[S_n - S_0] < 0$ as $n\to \infty$. Starting from $S_0 = 0$,  
the hitting time $\overline{\tau} := \min\{n > 0:  S_n = 0\}$ is a.s. finite, and  has a probability--generating function:
\begin{equation}
\label{eq:GF-hitting-time}
\mathbb E\left[s^{\tau}\right]=g(s),\quad 0\leq s\leq b^{-1/2}.
\end{equation}

Note that $b < 1$ since $\bar{p} <1/2$ under Assumption \ref{asmp:bounds}. 
 Since $\tau \ge 1$, the function $g(s)$ is increasing on 
$[0, b^{-1/2}]$;  its maximal value is achieved at  $s = b^{-1/2}$ and equals  
\begin{equation}\label{eqn-gb0.5}
g(b^{-1/2}) = \frac{1}{2\overline{p}b^{-1/2}} = \left[\frac{1 - \overline{p}}{\overline{p}}\right]^{1/2}.
\end{equation}
Accordingly, the function $G$ is defined for $u \in [0, u_*]$ with $u_* < \overline{q}$ solving the equation 
$\overline{q}/(\overline{q} - u_*) = b^{-1/2}$. This solution exists, and is unique and nonnegative. \hfill $\Box$
\label{rmk:g-max}
\end{remark}

 Now, let us compute
\begin{align*}
G(u^*) &= g(b^{-1/2}) = (2 \overline{q} b^{-1/2})^{-1} = \frac{\sqrt{\overline{p}(1 - \overline{p})}}{\overline{q} }\,, \\
\theta(\alpha, \overline{q}, \gamma, u^*) 
& = \frac{\overline{q}}{\overline{q}- u^*} - \frac{u^* \gamma}{( \overline{q}-u^*)( \overline{q}-u^*+\gamma)} \alpha^{ -( \overline{q}-u^*)/\gamma}    \\
&= \frac{1}{\sqrt{4 \overline{p}(1 - \overline{p})}} \left( 1- \frac{\sqrt{4 \overline{p}(1 - \overline{p})} +1}{\gamma^{-1}\overline{q} \sqrt{4 \overline{p}(1 - \overline{p})} +1 } \alpha^{-\gamma^{-1}\overline{q} \sqrt{4 \overline{p}(1 - \overline{p})}   } \right) \,.
\end{align*}

\begin{assumption} \label{AS-integrability} 
Recall that $\nu$ is the invariant measure for the 
process with the generator $\mathcal A$ on the state-space $\mathcal Z$. 
Assume the following integrability condition holds,  upon the measure $\nu$ and the cumulative birth-death ratio functions $r_n(z)$ 
(see \eqref{eqn-pi-n-BD}):
\begin{equation}\label{eq:modified-bounds1} 
 \sup_{u \in [0, u^*]} \sum_{n=0}^\infty  G^n(u)\int_{\mathcal{Z}} r_n(z) \nu(dz) <\infty. 
\end{equation}
\end{assumption} 

\begin{remark}
Recall the condition \eqref{eqn-Xi-finite-diff}, which is necessary for the existence of the invariant measure. 
 Condition \eqref{eq:modified-bounds1} is required in Step 1 of the proof of Theorem \ref{thm-RC-exp} below. 

For a single-server queue with the arrival and service rates  represented by $z=(z_1,z_2)
\in\mathcal Z$, where $\mathcal{Z}=\{(z_1,z_2)\in \RR^2_+: z_1<z_2\}$,  and
$r_n(z)=\rho (z)^{-n}$ where $\rho(z) = z_1/z_2$, 
condition \eqref{eq:modified-bounds1} requires that 
\begin{align*}
 \sum_{n=0}^\infty  G(u)^n \int_{\mathcal{Z}} \rho(z)^{-n} \nu(dz) =  \int_{\mathcal{Z}} \Big(1- \frac{G(u)}{z_2/z_1} \Big)^{-1} 
\nu(dz) <\infty,
\end{align*}
provided that $u \in [0, u^*]$ is such that $G(u) <z_2/z_1$ for $z=(z_1,z_2)\in\mathcal{Z}$.

For an infinite-server queue in Example \ref{exm:RBM-arrival},   condition \eqref{eq:modified-bounds1} 
can be written as
$$\sum_{n=0}^\infty\int_0^\infty\frac{(G(u)z/\mu)^n}{n!}\frac{2c}{\sigma^2} \exp\Big(- \frac{2c}{\sigma^2}z \Big)dz 
=  \frac{2c}{\sigma^2}\Big( \frac{2c}{\sigma^2}-G(u)/\mu\Big)^{-1}<\infty .$$
This holds for all $u \in [0,u^*]$ such that 
$$\dfrac{2c}{\sigma^2}-\dfrac{G(u)}{\mu} >0.$$

In Example \ref{exm:RBM-service},   condition \eqref{eq:modified-bounds1} is equivalent to  
$$\sum_{n=0}^\infty\int_{\mu_0}^\infty \frac{(\lambda G(u)/z)^n}{n!}\frac{2c}{\sigma^2} 
e^{-2c(z - \mu_0)/\sigma^2}dz 
= \frac{2c}{\sigma^2} \int_{\mu_0}^\infty e^{\lambda G(u)/z-2c(z - \mu_0)/\sigma^2}dz < \infty.$$
In  this model, $\lambda$ is any positive constant, so this also holds for any $u \in [0, u^*]$ since 
$\sup_{u \in [0, u^*]} \lambda G(u) <\infty$.

In Example \ref{example-MMinfty-2dRBM},   condition \eqref{eq:modified-bounds1}  is guaranteed when 
\begin{align*}
 & \sum_{n=0}^\infty G(u)^n  \int_0^\infty\int_{\mu_0}^\infty  \frac1{n!}\left(\frac{z_1}{z_2}\right)^n\alpha_1\alpha_2 
e^{-\alpha_1z_1 - \alpha_2(z_2-\mu_0)} dz_1 dz_2 \\ & = \int_0^{\infty}\int_{\mu_0}^{\infty}e^{G(u) z_1/z_2}
\alpha_1\alpha_2e^{-\alpha_1z_1 - \alpha_2(z_2-\mu_0)}dz_1dz_2 \\
& = \alpha_1\alpha_2 \int_{\mu_0}^{\infty}  \Big(  \int_0^{\infty}
e^{-(\alpha_1 - G(u)/z_2) z_1}dz_1 \Big)  e^{-\alpha_2 (z_2-\mu_0)}dz_2\\
& =   \alpha_1\alpha_2  \int_{\mu_0}^{\infty}
\frac{1}{\alpha_1- G(u)/z_2}  e^{-\alpha_2 (z_2-\mu_0)}dz_2\\
& <   \alpha_1\alpha_2  \int_{\mu_0}^{\infty}
\frac{1}{\alpha_1- G(u)/\mu_0}  e^{-\alpha_2(z_2-\mu_0)}dz_2\\
& = \frac{ \alpha_1\alpha_2}{\alpha_1- G(u)/\mu_0} \times \frac{1}{\alpha_2} =  \frac{ \alpha_1  }{\alpha_1- G(u)/\mu_0}<\infty. 
\end{align*}
This will require that $\mu_0 > \sup_{u \in [0, u^*]} G(u)/\alpha_1 $. Other models can be  treated similarly. \hfill $\Box$
\end{remark}

\begin{remark}
Inequality \eqref{eq:modified-bounds1} can be deduced from the following sufficient condition: 
\begin{equation}
\label{eq:modified-bounds}
\sup\limits_{z \in \mathcal Z}\sup_i\frac{\lambda_{i-1}(z)}{\mu_i(z)} \le \frac{\overline{p}}{1 - \overline{p}}\,.
\end{equation}
For example, for a single-server queue, 
the condition in \eqref{eq:modified-bounds} is equivalent to
$$
\sup_z \mu^{-1}(z)\Big(\lambda(z) -  \mu(z)\frac{\overline{p}}{1 - \overline{p}}  \Big)  \le 0.
$$
It indicates that, in addition to the stability condition $\lambda(z) <\mu(z)$,  we require that 
$\lambda(z) \le  \overline{p}\mu(z)/(1 - \overline{p})$ for all $z \in \mathcal Z$. For  a $K$-server queue in Example 
\ref{example-MMK-diff},  condition \eqref{eq:modified-bounds} is equivalent to
$$
\sup_z \sup_i \frac{\lambda(z) -  \mu(z)(i \wedge K)\overline{p}/(1 - \overline{p})}{\mu(z)(i\wedge K)}  \le 0. 
$$
It indicates that, in addition to the stability condition $\lambda(z) <\mu(z)K$,  we require that 
$\lambda(z) \le \overline{p}\mu(z)K/(1 - \overline{p})$ for all $z\in\mathcal Z$. 

\medskip

To check that condition \eqref{eq:modified-bounds} implies \eqref{eq:modified-bounds1},  we will show the following claim: 
There exists a constant $c \in (0, 1)$ such that for all $z \in \mathcal Z$ and $u \in (0,u^*]$,
\begin{equation}
\label{eq:bounded-series}
\varlimsup\limits_{n \to \infty}\frac{G^n(u)\cdot r_n(z)}{G^{n-1}(u)\cdot r_{n-1}(z)} \le c\,.
\end{equation}
Then, combining~\eqref{eq:bounded-series} with the observation that the term corresponding to $n = 0$ is $G(u)^0r_0(z) = 1$, 
we get  that the series inside the integral  in~\eqref{eq:modified-bounds1}
is estimated from above by $1 + c + c^2 + \ldots = 1/(1 - c)$. Integrating this with respect to the probability measure $\nu$, we 
complete the derivation of~\eqref{eq:modified-bounds1}.

Let us now verify  the claim in ~\eqref{eq:bounded-series}. We can replace $G(u) = g(s)$ for $s = \overline{q}/(\overline{q} - u)$. 
Using~\eqref{eqn-pi-n-BD}, we can replace
\begin{equation}
\label{eq:fracs}
\frac{r_n(z)}{r_{n-1}(z)} = \frac{\lambda_{n-1}(z)}{\mu_n(z)}\,.
\end{equation}
Combining~\eqref{eq:fracs} with~\eqref{eq:modified-bounds} and \eqref{eqn-gb0.5} in Remark~\ref{rmk:g-max}, we get 
that the left-hand side of~\eqref{eq:bounded-series} is bounded by $c = [\overline{p}/(1 - \overline{p})]^{3/2}$. \hfill $\Box$
 \end{remark}

Here is the result on exponential convergence in the first scenario, extending \cite[Theorem 4.1]{PSBS}. 
\begin{theorem} \label{thm-RC-exp}
Suppose that  Assumptions \ref{as-diff}, \ref{asmp:bounds}, \ref{asmp:basic}, and \ref{AS-integrability} hold true. Then
 there exists a constant $C  > 0$ such that for all $n \in \NN$ and $z \in \mathcal Z$,  the transition probabilities for the Markov process
$(N,Z)$ satisfy
\begin{align}
\label{eq:TV-conv}
\|P^t((n, z), \cdot) - \pi(\cdot)\|_{\mathrm{TV}} \le C(1 + G(u)^{n})e^{-\varkappa t}\,.
\end{align}
Here $\pi(\cdot)$ is given in \eqref{eqn-pi-diff} and $ \varkappa = (1 - \varepsilon)u$, where 
 $\varepsilon \in (0, 1)$, and $u \in (0, u^*]$ satisfy
\begin{equation}\label{eq:condition}
G(u)\theta(\alpha, \overline{q}, \gamma, u) < \left(1 - \alpha^{-\overline{q}/\gamma}
\frac{\gamma}{\overline{q}+\gamma}\right)^{-\varepsilon/(1 - \varepsilon)} \,.
\end{equation} 
\end{theorem}

\begin{remark}
The  bound \eqref{eq:condition} holds for some $u$ in  a right  neighborhood of zero. 
Indeed, its left-hand side is continuous in $u$, and its right-hand side is independent of $u$.  Next, the left-hand side value 
at $u = 0$ is equal to $1$ (since $G(1) = g(0) = 1$, and $\theta(\alpha, \overline{q}, \gamma, 0) = 1$). The right-hand side is 
greater than $1$ (since $-\varepsilon/(1 - \varepsilon) < 0$).  \hfill $\Box$
\end{remark}

\begin{proof} 
We modify the proof of \cite[Theorem 4.1]{PSBS}  and rectify the coupling argument for the joint Markov process $(N,Z)$. The general idea of coupling is to take two copies $(N_1,Z_1)$ and 
$(N_2,Z_2)$ of the process $(N,Z)$ starting from initial states $(n_1,z_1)$ and $(n_2,z_2)$ and run them jointly, achieving the coupling time $\btau=\btau_{(n_1,z_1),(n_2,z_2)}$ with $ \mathbb{E}[e^{\varkappa \btau}]<\infty$ for some constant $\varkappa >0$. Here and below,
\begin{equation}
\label{eq:btau}
\btau :=\inf\{t\ge 0: N_1(t) = N_2(t), Z_1(t) = Z_2(t)\}.
\end{equation}
Then Lindvall's inequality \cite{lindvall1979note} can be applied to obtain the bound
$$
\|P^t((n_1, z_1), \cdot) - P^t((n_2, z_2), \cdot)\|_{\mathrm{TV}}\le 
\mathbb{E}\left[e^{\varkappa \btau}\right] e ^{-\varkappa t}.
$$
To produce the coupling, we need to first wait for the birth-death components $N_1$, $N_2$ being coupled 
at state $n=0$ and then wait until either the environment processes are coupled or the birth-death processes have a birth in which case the coupling process restarts anew. The main task is to estimate  the expected value $\mathbb{E}[e^{\kappa\btau}]$ for the coupling time $\tau$ and 
identify the exponent $\kappa$. The proof is divided into seven steps, similarly to \cite[Theorem 4.1]{PSBS}, but some steps require substantial changes as highlighted below. 

\medskip

{\it Step 1.} Observe that to prove \eqref{eq:TV-conv}, it suffices to show
$$
\|P^t((n_1, z_1), \cdot) - P^t((n_2, z_2), \cdot)\|_{\mathrm{TV}}\le  (G(u)^{n_1} + G(u)^{n_2}) e ^{-\varkappa t}.
$$
This is because we can obtain \eqref{eq:TV-conv} from the above by using the definition of the total variation distance, and integrating with respect to $(n_2,z_2)\sim \pi$; integrability is guaranteed under Assumption \ref{AS-integrability}. In turn, it suffices to show that   
$$
\|P^t((n_1, z_1), \cdot) - P^t((n_2, z_2), \cdot)\|_{\mathrm{TV}}\le  G(u)^{n_1\vee n_2} e ^{-\varkappa t}.
$$

\medskip

{\it Step 2.}  In view of Lindvall's inequality \cite{lindvall1979note}, we need to prove that for the coupling time $\btau$,  we have $\mathbb{E}\left[e^{\varkappa\btau}\right]<\infty$, so that 
\begin{equation}
\label{eqn-step2}
\|P^t((n_1, z_1), \cdot) - P^t((n_2, z_2), \cdot)\|_{\mathrm{TV}}\le 2 \mathbb{P}(\btau>t) \le 2 \mathbb{E}\left[e^{\varkappa \btau}\right] e^{-\varkappa t}.
\end{equation} 

\medskip

{\it Step 3.} The coupling time $\btau$ can be written in terms of auxiliary random variables
$\tau_j$, $\eta_j$, $\zeta_j$ and a stopping time $\mathcal J$. Formally:
\begin{equation}
 \label{eqn-tau-1}
\btau =\sum_{j=0}^{\mathcal J}\tau_j+\zeta_{\mathcal J}\,\,.
\end{equation}
Here the initial random times are defined as 
\begin{align*}
\tau_0&:= \inf\{t \ge 0\mid N_1(t) = N_2(t) = 0\},\\ 
\eta_0&:=\inf\{t\ge 0: N_1(\tau_0+t)\vee N_2(\tau_0+t)=1\},\\
\zeta_0&:=\inf\{t\ge 0: Z_1(\tau_0+t) = Z_2(\tau_0+t)\},
\end{align*}
where $\overline{\lambda} = \sup_z \sup_i \lambda_i(z)$.
 The subsequent random times are defined iteratively: for $j\ge 1$,
 \begin{align*}
&\tau_j:= \inf\left\{t\ge\eta_{j-1}: N_1\left(\sum_{0\leq l\leq j-1}\tau_{l}+t\right)
=N_1\left(\sum_{0\leq l\leq j-1}\tau_{l}+t\right)=0\right\},\\
&\eta_j:= \inf\left\{t\ge 0: N_1\left(\sum_{0\leq l\leq j}\tau_{l}+t\right)\vee N_2\left(\sum_{0\leq l\leq j}\tau_{l}+t\right)=1\right\},\\
&\zeta_j:= \inf\left\{t\ge 0: Z_1\left(\sum_{0\leq l\leq j}\tau_{l}+t\right) = Z_2\left(\sum_{0\leq l\leq j}\tau_{l}+t\right)\right\}.
\end{align*}
Finally, the stopping time $\mathcal J$ is defined by
\begin{equation}\label{eqn-stoppingJ}
\mathcal{J} = \min \{j \ge 0: \zeta_j < \eta_j\}.
\end{equation}
Observe that if $\zeta_0 <\eta_0$, then $\tau_0+\zeta_0$ gives the coupling time $\btau$ for $(N_1,Z_1)$ and $(N_2,Z_2)$. Otherwise, if $\zeta_0>\eta_0$ and $\zeta_1<\eta_1$, then $\tau_0+\tau_1+\zeta_1$ gives the coupling time, and so on. The procedure continues until the coupling time  $\btau$ from \eqref{eqn-tau-1}. 
  
\medskip

{\it Step 4.} Here we show that for processes $(N_1,Z_1)$ and $(N_2,Z_2)$,
\begin{equation}
\label{eq:MGF}
\mathbb E[e^{u\tau_k}] \le  G(u)^{n_1\vee n_2},  \quad k =0, 1, \dots
\end{equation}
First, let us study the case of $\tau_0$. Consider an embedded discrete-time Markov chain 
$\left\{N^*_z(l),\,l\geq 0\right\}$ 
on the state-space $\mathbb N$ with the transition probability from $i$ to $i+1$ given by $p_i(z)$ 
and that from $i$ to $i-1$ given by $1 - p_i(z)$. Cf. \eqref{eq:pbarz}. Such a chain is dominated by the  Markov chain $\left\{\,{\ov{N^*}}(l),\,l\geq 0\right\}$ on $\mathbb N$ with the transition probabilities of jumps up and down $\overline{p}$ and $1 - \overline{p}$. (From the left-most state $n=0$ both chains jump up with probability $1$.)

More precisely, let $V := \min\{l \ge 0\mid{\ov{N^*}}(l) = 0\}$. Assumption~\ref{asmp:bounds} implies the following stochastic domination:
\begin{equation}
\label{eq:domination}
\begin{array}{l}
\tau_0 \preceq U_1 + \ldots + U_V\\
\quad\hbox{where}\;U_i \sim \mathrm{Exp}(\overline{q})\;\hbox{are i.i.d. random variables
independent of $V$.}
\end{array} 
\end{equation}
Indeed, the waiting times for jumps in the processes $N_1$ and $N_2$ have intensity at least $\overline{q}$, and in both processes, the embedded chain  $\left\{N^*_z(l),\,l\geq 0\right\}$ has a stronger drift towards the state $n=0$ than the chain $\left\{\,{\ov{N^*}}(l),\,l\geq 0\right\}$. For $U_i$, the MGF is $\mathbb{E}[e^{u U_i}] = \overline{q}/(\overline{q} - u)$. Also, 
$$
\mathbb{E}\Bigl[e^{u ( U_1 + \ldots + U_V) }\Bigr] = \mathbb{E}\Bigl[ (\mathbb{E}[e^{u U_1}])^V\Bigr] = \mathbb{E} \left[\bigg(\frac{\overline{q}}{\overline{q} - u}\bigg)^V \right].
$$
Combining this with~\eqref{eq:GF-hitting-time}, we get the estimate ~\eqref{eq:MGF} for $\tau_0$. 
For $\tau_i$ with $i\geq 1$ we apply a similar argument and get inequality \eqref{eq:MGF} with $n_1\vee n_2$ replaced by $1$.

\medskip

{\it Step 5.} Use the same argument as in Step 5 of the proof of  \cite[Theorem 4.1]{PSBS}. Apply Assumption \ref{asmp:basic} 
and  Lemma  \ref{lemma:tech} from Appendix B. As a result, the random variable $\mathcal{J}$ is stochastically dominated by a geometric 
random variable $\overline{\mathcal{J}}$ with parameter 
$$\overline\vartheta= \frac{\gamma}{\overline{\lambda}+\gamma} \alpha^{-\overline{\lambda}/\gamma}\,,$$ 
which has the probability-generating function  
$$
\mathbb{E} \Big[s^{\overline{\mathcal{J}}}\Big]= \frac{\overline{\vartheta} s}{1-(1-\overline{\vartheta})s}\,,\quad s \in [0, 1/(1-\overline{\vartheta})).
$$
Recall, $\gamma >0$ is the value introduced in \eqref{eq:gamma}.

\medskip

{\it Step 6.} Consequently, under Assumption \ref{asmp:basic}, by  Lemma \ref{lemma:tech} we obtain 
$$
\mathbb{E} \Big[e^{u (\zeta_k \wedge \eta_k )}\Big] \le \theta(\alpha, \bar{q}, \gamma,  u), \quad k =0,1,\dots
$$

\medskip

{\it Step 7.} Finally we derive the upper bound for the MGF of the coupling time $\btau$. We have 
from Steps 4 and 6 that
$$
\psi(u):=\mathbb{E} \Big[e^{u (\tau_k+\zeta_k \wedge \eta_k ) }\Big] \le G(u)^{n_1\vee n_2} \theta(\alpha, \bar{q}, \gamma,  u) .  
$$
By applying the optimal stopping theorem to the martingale 
$$
M_\ell= \exp\bigg(u \sum_{k=0}^{\ell} (\tau_k+\zeta_k \wedge \eta_k ) - \ell \ln \psi(u)\bigg),\;
\ell =0,1,\ldots ,
$$ 
and the stopping time $\mathcal{J}$, we obtain that
 $$
\mathbb{E} [M_{\mathcal{J}}] = \mathbb{E} [M_0] =\mathbb{E} [e^{u  (\tau_0+\zeta_0 \wedge \eta_0 ) }] \le G(u) 
\theta(\alpha, \bar{q}, \gamma,  u). 
$$
Now, we have
$\mathbb{E} [M_{\mathcal{J}}]  = \mathbb{E} \big[\exp\big(u (\btau - \mathcal{J} \ln \psi(u)) \big)\big] $. 
Using H{\"o}lder's inequality, we get
$$ 
\mathbb{E} \big[\exp\big( (1-\epsilon) u\btau\big)\big]  \le ( \mathbb{E} [M_{\mathcal{J}}])^{1-\epsilon}   
 \cdot\left(\mathbb{E} [ \psi(u)^{(1-\epsilon) \mathcal{J}/\epsilon}]\right)^\epsilon.
$$
Here $\epsilon\in (0,1)$ is the value from \eqref{eq:condition}. Using the result in Step 5 with 
$s=\psi (u)^{(1-\eps )/\eps}$, we obtain inequality
\eqref{eq:TV-conv}.
\end{proof}

\subsection{The second scenario} 

We next consider the second scenario, starting with the following assumption.

\begin{assumption} \label{asmp:bounds-s2}
In addition to \eqref{eq:main-bounds>0}, assume that there exist constants $\bar\lambda>0$ and $\bar\mu>0$ such that $\lambda_n(z) \le \bar\lambda$ and $\mu_n(z) \ge n \bar\mu$ for all $z$ and $n$. 
\end{assumption}

\begin{remark}
Assumption \ref{asmp:bounds-s2} applies to the models M/M/$\infty$, M/M/$K$/$0$ and M/M/$K+$M (cf. Examples \ref{example-MMK+M} and \ref{example-MMK+M-diff}) which do not require a stability condition like \eqref{eq:main-bounds}. The idea is to use an infinite-server queue to dominate the process $N$ in these models. For example,  in Example \ref{example-MMK+M-diff},  one can assume that there exist constants 
$\bar\lambda$ and $\bar\mu$ such that for all $n$, 
$$
\sup_z \lambda_n(z) \le \bar\lambda, \quad \inf\mu_n(z) \ge \overline{\mu}n.
$$
(Note that $ \inf_z\big\{ \mu(z) (n \wedge K) + (n-K)^+ \gamma(z) \big\} \ge \inf_z\big\{ \mu(z) \wedge \gamma(z)\big\} \big[ (n \wedge K) + (n-K)^+\big] = \inf_z\big\{ \mu(z) \wedge \gamma(z)\big\}  n$.)
Then the M/M/$K+$M queueing process is dominated by the birth-death process with birth and death rates $\bar\lambda$ and $\bar\mu n$, respectively. \hfill $\Box$
\end{remark}

We will expand on the last remark and refer to the birth-death process with rates $\overline{\lambda}$ and $n\overline{\mu}$ as $\overline{N} = (\overline{N}(t),\, t \ge 0)$, assuming that $\overline{N}(0) = 1$. In other words, $\overline{N}$ is a  M/M/$\infty$ queueing process; it will be used for the purpose of stochastic domination in the proof of Theorem~\ref{thm-RC-exp2} below. More precisely, consider the busy period  
\begin{equation}
    \label{eq:busy}
\overline{\tau} := \inf\{t >0:\, {\ov N}(t) = 0\}.
\end{equation}
The MGF ${\ov G}(u) = \mathbb E[e^{u\ov\tau}]$ is finite for all $u>0$ and  can be expressed by using Kummer's function \cite[Proposition 4.1]{guillemin1995transient} (see also \cite[Theorem 1]{Busy}).  The explicit expression for $\overline{G}(u)$ is not essential for our results and omitted for brevity.

 \begin{theorem} \label{thm-RC-exp2} 
Under Assumptions \ref{as-diff},  \ref{asmp:basic}, \ref{AS-integrability} and \ref{asmp:bounds-s2}, the results in Theorem \ref{thm-RC-exp} hold with ${\ov G}(u) = \mathbb E[e^{u{\ov\tau}}]$. This includes \eqref{eq:condition} with $\varepsilon \in (0,1)$, $u >0$, and $\ov G$ in place of $G$. 
\end{theorem}

\begin{proof} 
We follow the same steps as in the proof of Theorem \ref{thm-RC-exp}. The arguments in Steps 1--3 and 5--7 remain valid without changes. In Step 4, we consider the  MGF ${\ov G}(u)$: it is instrumental as the above process $\overline{N}$ can be used to dominate components $N_1$ and $N_2$ in the processes $(N_1, Z_1)$ and $(N_2, Z_2)$ with initial states $(n_1,z_1)$ and $(n_2,z_2)$ (cf. \eqref{eq:btau}). As before, let $\btau=\btau_{(n_1,z_1),(n_2,z_2)}$ denote the coupling time  for processes $(N_1,Z_1)$ and $(N_2,Z_2)$. Suppose that $\btau$ satisfies
\begin{equation}
\label{eq:MGF-coupling}
\mathbb E[e^{u\btau}] \le{\ov G}(u)^{n_0},\;\hbox{ where}\;n_0=n_1\vee n_2.
\end{equation} 
Then the rest of the argument is completed as in Theorem \ref{thm-RC-exp}. Hence, we focus on the proof of~\eqref{eq:MGF-coupling}. To this end, given $n$, we set 
\begin{align*}
\tau_0 &= \inf\{t \ge 0\mid \overline{N}(t) = n\};\\
\tau_k &= \inf\{t \ge \tau_{k-1}\mid \tilde{N}(t+\tau_{k-1}) = n-k\}\quad \mbox{for}\quad k = 1, \ldots, n.
\end{align*}
Then $\tau_n \equiv \overline{\tau}$. Let us show that for all $k = 1, \ldots, n$,
\begin{equation}
\label{eq:MGF-induction}
\mathbb E[e^{u\tau_k}] \le{\ov G}(u)^k.
\end{equation}
Use induction by $k$. For $k = 0$, the bound holds trivially. For the induction step: given  $k\ge 1$, 
\begin{equation}
\label{eq:telescope}
\mathbb E[e^{u\tau_k}] = \mathbb E[\mathbb E[e^{u\tau_k}\mid \tau_{k-1}]] = 
\mathbb E[e^{u\tau_{k-1}}\mathbb E[e^{u(\tau_k - \tau_{k-1})}\mid \tau_{k-1}]].
\end{equation}
Consider the process $\widetilde{N} = \{\widetilde{N}(t),\, t \ge 0\}$ where $\widetilde{N}(t) = \overline{N}(t + \tau_{k-1}) - (n-k)$. If we let it run only until the time $t = \tau_k - \tau_{k-1}$, this behaves as a birth-death process with birth rate $\overline{\lambda}$ and death rate $\overline{\mu}(n-k+m)$ at state $m$, starting from $\widetilde{N}(0) = 1$. The hitting time of $0$ by the process $\widetilde{N}$ (that is, the busy time) is $\tau_k - \tau_{k-1}$. Next, the birth rates of this new process $\widetilde{N}$ coincide with the birth rates of $\overline{N}$ whereas the death rates of $\widetilde{N}$ are at least as large as in $\overline{N}$. Thus, $\tau_k - \tau_{k-1}$ is stochastically dominated by $\overline{\tau}$, the busy period in $\overline{N}$ whose MGF is $G(u)$, that is,  $\tau_k - \tau_{k-1} \preceq \overline{\tau}$. Therefore, for $u > 0$, 
\begin{equation}
\label{eq:domination-MGF}
\mathbb E[e^{u(\tau_k - \tau_{k-1})}\mid \tau_{k-1}] \le G(u).
\end{equation}
Suppose that~\eqref{eq:MGF-induction} is true for $k-1$ instead of $k$. Combining this assumption with~\eqref{eq:telescope} and~\eqref{eq:domination-MGF}, we get~\eqref{eq:MGF-induction} for $k$. This completes the proof of~\eqref{eq:MGF-induction} and with it the proof of~\eqref{eq:MGF-coupling}. 
\end{proof}

\section{Polynomial Convergence to Stationarity}
\label{sec-1/t}

In this section, instead of the conditions on the birth and death rates in Assumptions \ref{asmp:bounds} and \ref{asmp:bounds-s2}, we make the following Assumption \ref{asmp:polynomial}  such that the joint Markov process $(N,Z)$ has a polynomial rate of convergence, while the generator of environment $\mathcal{A}$ satisfies Assumption \ref{asmp:basic} with an exponential rate of convergence. 

Subexponential/subgeometric convergence has been studied less than geometric. Nevertheless, over the last few decades it has amassed a substantial literature. In  \cite{Fort2005, Fort2009} it was studied for general continuous time Markov processes, using a modified Lyapunov condition: $\mathcal L V \le -\varphi(V)$ for a  sublinear function $\varphi$ (here $\mathcal L$ is the generator of this Markov process). 
%In contrast, a standard Lyapunov condition with a linear $\varphi(s) = ks$ implies geometric convergence. 
Continuous-time Markov chains on the state space $\mathbb N$ with subgeometric convergence were studied in \cite{subgeometric-CTMC}, including applications to birth-death process. Another related article  is \cite{ergodic-degrees}, giving estimates for hitting time moments. A slightly different approach was taken in \cite{Butkovsky}. 
Combining the approach from \cite{Fort2005, Fort2009} with stochastic ordering property, subgeometric convergence results  are also recently developed in \cite{sarantsev2021sub}. See also the  subexponential upper and lower bounds in Wasserstein distance for general Markov processes in \cite{sandric2022subexponential}.

\begin{assumption} 
\label{asmp:polynomial}
There exist constants $\bar\lambda_n$ and $\bar\mu_n$ such that $\lambda_n(z) \le \bar\lambda_n$ and $\mu_n(z) \ge \bar\mu_n$ for each $(n, z) \in \mathbb N\times\mathcal Z$. 
Moreover, there exist a nondecreasing function $V : \mathbb N \to [0, \infty)$ with $V(0) = 0$ and a constant $\overline{C}_V > 0$  such that the generator $\overline{\mathcal L}$ of the birth-death process $\overline{N} = \{\overline{N}(t), t \ge 0\}$ with rates $\overline{\lambda}_n$ and $\overline{\mu}_n$ satisfies
\begin{equation}
\label{eq:Lyapunov-1/t}
\overline{\mathcal L}V(n) := \overline{\lambda}_n(V(n+1) - V(n)) + \overline{\mu}_n(V(n-1) - V(n)) \le -\overline{C}_V,\quad n \ge 1.
\end{equation}
\end{assumption}

\begin{remark}
It is clear that if $\lambda_n(z) \le \bar\lambda_n$ and $
\mu_n(z) \ge \bar\mu_n$ for each $z$ and $n$, then the process $N$ can be dominated by   $\overline{N}$ described above. 
Taking $V(n) = n+1$,  the condition  in \eqref{eq:Lyapunov-1/t} becomes $\overline{\lambda}_n - \overline{\mu}_n \le -\overline{C}_V$. We can rewrite it as $\inf\limits_{n \ge 1}(\overline{\mu}_n - \overline{\lambda}_n) > 0.$
For example, consider $\lambda_n(z) = \overline\lambda(z)$ and $\mu_n(z) = \overline\mu(z) + n^{1/2}$ with $\overline\lambda(z) \equiv \overline\mu(z)$ for all $z$. (This can be regarded as a single-server queue with state-dependent service rates.)
\hfill $\Box$
\label{rmk:simple}
\end{remark}

In the following lemma, we drive a bound for the expected hitting time for the dominating process under the conditions in Assumption \ref{asmp:polynomial}.
% and show that the convergence to stationarity in the total variation distance is at least $1/t$.
The proof of the lemma is postponed until the end of this section.  

\begin{lemma} 
\label{lemma:1/t}
%Assume that the generator of process $\overline{N}$ with generator $\overline{\mathcal L}$ satisfies~\eqref{eq:Lyapunov-1/t}. 
Let  $\overline{P}^t(n, \cdot)$ denote the transition function of the dominating process $\overline{N}$ in Assumption \ref{asmp:polynomial}. 
The hitting time $\overline{\tau}_0 := \inf\{t \ge 0 \,|\,  \overline{N}(t) = 0\}$ of $0$ satisfies 
$$\mathbb E_n[\overline{\tau}_0\mid \overline{N}(0) = n] \le \frac{V(n)}{\overline{C}_V}\, .$$ 
The  process $\overline{N}$ has a unique stationary distribution $\overline{\pi}$. Furthermore, $\mathbb E_{\overline{\pi}}[V] < \infty$, and 
\begin{equation}
\label{eq:dominating-1/t}
\|\overline{P}^t(n, \cdot) - \overline{\varkappa}(\cdot)\|_{\mathrm{TV}} \le \frac{2(\mathbb E_{\overline{\pi}}[V] + V(n))}{\overline{C}_V}\times \frac{1}{t}\,,\quad t > 0,\quad n \in \mathbb N.
\end{equation}
\end{lemma}

Now we state the main result on polynomial convergence.

\begin{theorem}
\label{thm:polynomial}
Under Assumptions \ref{as-diff}, \ref{asmp:basic} and \ref{asmp:polynomial}, there exists a constant $C$ such that for joint process $(N, Z)$, we have
\begin{equation}
\label{eq:polynomial}
\|P^t((n,z), \cdot) - \pi(\cdot)\|_{\mathrm{TV}} \le \frac{C(1 + V(n))}t,\quad t > 0,\quad (n,z) \in\NN\times \mathcal{Z}.
\end{equation}
\end{theorem}

%\begin{remark}
%The convergence rate for $X = (N, Z)$ is $1/t$, which is the slower of the rates for the two marginals. Cf. Remark~\ref{rmk:slowest-general}.
%\label{rmk:slowest}
%\end{remark}

\begin{proof} We again follow the approach from \cite{PSBS} with a seven-step proof, as in Theorems~\ref{thm-RC-exp} and~\ref{thm-RC-exp2}. Here we estimate the expectation of the coupling time directly. Under Assumption~\ref{asmp:polynomial}, the component $N$ in $(N, Z)$ is dominated by $\overline{N}$. We modify Steps 1--7 from the proof of Theorem~\ref{thm-RC-exp} as follows.

\medskip

{\it Step 1:} Here we show that the stationary distribution $\pi$ for the process $(N, Z)$ satisfies $\mathbb E_{\pi}[V] < \infty$. Indeed, by Assumption~\ref{asmp:polynomial}, this stationary distribution $\pi$ for the component $N$ is stochastically dominated by $\overline{\pi}$, the stationary distribution of the process $\overline{N}$. Consequently, $\mathbb E_{\pi}[V] \le \mathbb E_{\overline{\pi}}[V]$ and $\mathbb E_{\overline{\pi}}[V] < \infty$ by Lemma~\ref{lemma:1/t}. Next, we must prove the version of~\eqref{eq:polynomial} with two starting points:
\begin{equation}
\label{eq:polynomial-coupled}
\|P^t((n_1,z_1), \cdot) - P^t((n_2, z_2), \cdot)\|_{\mathrm{TV}} \le \frac{C(V(n_1) + V(n_2))}{t}\,.
\end{equation}
Similarly to Step 1 in the proof of Theorem~\ref{thm-RC-exp}, the bound $\mathbb E_{\pi}[V] < \infty$, together with~\eqref{eq:polynomial-coupled},  gives us the required result~\eqref{eq:polynomial}. 

\medskip

The next six steps are devoted to the proof of~\eqref{eq:polynomial-coupled}. For brevity, in each step we only highlight the changes compared to the same step in the proof of Theorem~\ref{thm-RC-exp}.

\medskip

{\it Step 2:} Instead of \eqref{eqn-step2}, we use the Markov inequality for the coupling time $\btau = \btau_{(n-1, z_1), (n_2, z_2)}$:
\begin{equation}\label{eqn-step2-poly}
\|P^t((n_1, z_1), \cdot) - P^t((n_2, z_2), \cdot)\|_{\mathrm{TV}}\le 2 \mathbb{P}(\btau>t) \le  \frac{2 \mathbb{E}[\btau] }{t}, \quad t>0. 
\end{equation}

{\it Step 3:} Here we obtain the same expression for the coupling time 
$\tau$ as in \eqref{eqn-tau-1}. 

{\it Step 4:} We use random variables $\tau_j$, $\zeta_j$, $\eta_j$ as defined earlier. Instead of estimating the MGF, we get estimates for the mean of the stopping time, and this serves as the backbone of our proof. The required property is already established in Lemma~\ref{lemma:1/t}: $\mathbb E[\tau_0] \le V(n_1\vee n_2)/\overline{C}_V$ for $j = 0$, and similarly with $n_1\vee n_2$ replaced by $1$ for $j \ge 1$: $\mathbb E[\tau_j] \le V(1)/\overline{C}_V$. Combining these estimates, we get:
\begin{equation}
    \label{eq:combined-estimate-mean}
\mathbb E[\tau_j] \le \frac{V(n_1\vee n_2\vee 1)}{\overline{C}_V}\,.
\end{equation}

{\it Step 5:} Recall the stopping time $\mathcal{J}$ in \eqref{eqn-stoppingJ}. Note that it is stochastically dominated by a geometric random variable $\overline{\mathcal{J}}$ with parameter $\overline\vartheta= \mathbb{P}(\zeta_k < \eta_k)$. By  Lemma \ref{lemma:tech}, we get 
\begin{equation}
\label{eq:p-estimate}
\overline\vartheta = \frac{\gamma}{\overline{\lambda}+\gamma}\alpha^{-\overline{\lambda}/\gamma}\,.
\end{equation}
Thus, the expectation of $\overline{\mathcal{J}}$ is given by 
\begin{equation}
\label{eq:est-J}
\mathbb{E}[\overline{\mathcal{J}}] = (1- \overline\vartheta) \overline\vartheta^{-1} \le \overline\vartheta^{-1} =  \frac{\overline{\lambda}+\gamma}{\gamma}\alpha^{\overline{\lambda}/\gamma}\,.
\end{equation}

{\it Step 6:} Continuing Step 4, we get estimates for the mean:
\begin{equation}
\label{eq:exp-estimate}
\mathbb E\left[\zeta_j\wedge\eta_j\right] \le \mathbb E[\eta_j] \le \overline{\lambda}^{-1}.
\end{equation}
Let $\xi_j =\tau_j + \zeta_j\wedge\eta_j$. By Assumption \ref{asmp:polynomial} and \eqref{eq:exp-estimate}, combined 
with~\eqref{eq:combined-estimate-mean}, we have 
\begin{equation}
\label{eq:mean-estimate}
\mathbb{E}[\xi_k] = \mathbb{E}\left[\tau_k+ \zeta_k\wedge\eta_k\right] \le \mathbb E[\tau_j] +  \overline{\lambda}^{-1} \le \overline{C}_\xi := \frac{V(n_1\vee n_2\vee 1)}{\overline{C}_V} + \overline{\lambda}^{-1} <\infty,\quad k \ge 1.
\end{equation} 

{\it Step 7:} We have the following estimates of the expectation of the coupling time. Define the random walk $S_k := \xi_0 + \xi_1 + \ldots + \xi_k$, $k \ge 0$. Then its centered version $M_k = S_k - \mathbb E[S_k]$, $k\ge 0$, is a martingale. From the estimate~\eqref{eq:mean-estimate}, we get that the process
$
\{S_k - k\overline{C}_\xi:\, k \ge 0\},
$
is a supermartingale. Applying the optional stopping theorem with the stopping time $\mathcal{J}$, we  get
$$
\mathbb E[S_{\mathcal{J}}] \le \mathbb E[\mathcal J]\cdot\overline{C}_\xi + \mathbb E[S_0].
$$
Combining this estimate with~\eqref{eq:est-J} and \eqref{eq:mean-estimate}, we get
\begin{align*}
\mathbb{E}[S_{\mathcal{J}}] \le \overline{C}_\xi \cdot \left(1+ \frac{\overline{\lambda}+\gamma}{\gamma}\alpha^{\overline{\lambda}/\gamma} \right). 
\end{align*}
Applying Lindvall's inequality from \cite[Chapter 1]{lindvall} we complete the proof of~\eqref{eq:polynomial-coupled}, and with it the proof of the main result~\eqref{eq:polynomial}. 
\end{proof}

\begin{proof}[Proof of Lemma~\ref{lemma:1/t}] 

We adapt the proof of \cite[Theorem 1]{sarantsev2021sub}. Assume that $\overline{N}(0) = n$. From the condition~\eqref{eq:Lyapunov-1/t}, we get that the process $Y(t) = \overline{C}_V(t\wedge\overline{\tau}_0) + V(\overline{N}(t\wedge\overline{\tau}_0))$ is a nonnegative local supermartingale. By Fatou's lemma, it is a true supermartingale. Then by the optional stopping theorem: $\mathbb E[Y(\overline{\tau}_0)] \le \mathbb E[Y(0)] = V(\overline{N}(0)) = V(n)$. Next, $Y(\overline{\tau}_0) = \overline{C}_V\overline{\tau}_0 + V(\overline{N}(\overline{\tau}_0)) = \overline{C}_V\overline{\tau}_0$. Combining these observations, we complete the proof that 
\begin{equation}
\label{eq:bound-tau}
\mathbb E[\tau_0] \le \frac{V(n)}{\overline{C}_V}\,.
\end{equation} 
Next, $\overline{P}^t(x, y) > 0$ for all $x, y \in \mathbb{N}$ and $t > 0$. Apply the classification adopted in \cite{meyn1993stabilityII} on transient, null recurrent, and positive recurrent processes with counting reference measure. The singleton $\{0\}$ has a positive reference measure. Therefore, the process $\overline{N}$ is positive recurrent. Consequently, $\overline{N}$ has a unique stationary distribution $\overline{\pi}$, and we have convergence $\|\overline{P}^t(n, \cdot) - \overline{\pi}(\cdot)\|_{\mathrm{TV}} \to 0$ for any initial state $n$. Recall the condition $\overline{\mathcal L}V \le -\overline{C}_V$. It holds for all $n \in \mathbb{N}$, except $n = 0$. Using the terminology of \cite{meyn1993stabilityII}, the singleton $\{0\}$ is a small set. Thus, $\mathbb E_{\overline{\pi}}[V] < \infty$. 

\medskip

Consider now two versions $\overline{N}_1$ and $\overline{N}_0$ of the process $\overline{N}$, one starting from the state $n$, and the other from the stationary distribution $\overline{\pi}$ (the stationary version). Both are dominated by the version $\overline{N}_*$ starting from $n\vee \overline{n}^*$ where $\overline{n}^* \sim \overline{\pi}$:
$$
\overline{N}_0(t) \le \overline{N}_*(t)\quad \mbox{and}\quad \overline{N}_1(t) \le \overline{N}_*(t)\quad \mbox{for all}\quad t \ge 0.
$$
Hence, for the hitting time $\overline{\tau}_* := \inf\{t \ge 0\mid \overline{N}_*(t) = 0\}$ we also have $\overline{N}_0(\overline{\tau}_*) = \overline{N}_1(\overline{\tau}_*) = 0$. Couple the processes $\overline{N}_0$ and $\overline{N}_1$ so that $\overline{N}_0(t) = \overline{N}_1(t) = 0$ for $t > \overline{\tau}_*$. By a standard coupling argument, 
\begin{align}
    \label{eq:final-ineq}
    \begin{split}
     \|\overline{P}^t(n, \cdot) - \overline{\pi}(\cdot)\|_{\mathrm{TV}}  & = \sup\left\{\big|\mathbb P\big(\overline{N}_0(t) \in A\big) - \mathbb P\big(\overline{N}_1(t) \in A\big)\big|: A \subseteq \mathbb{N}\right\} \\ & \le 2\mathbb P(\overline{\tau}_* > t) \le \frac{2\mathbb E[\overline{\tau}_*]}{t}\,.
\end{split}
\end{align}
Here we used stationarity of the process $\overline{N}_1$: for every $t \ge 0$, $\overline{N}_1(t) \sim \overline{\pi}$. Combining~\eqref{eq:final-ineq} with~\eqref{eq:bound-tau} yields
\begin{equation}
    \label{eq:final-estimate}
\|\overline{P}^t(n, \cdot) - \overline{\pi}(\cdot)\|_{\mathrm{TV}} \le \frac{2\mathbb E[V(n\vee\overline{n}^*)]}{C_Vt}\,,\quad \overline{n}^* \sim \overline{\pi}.
\end{equation}
Finally, let us estimate $\mathbb E[V(n\vee\overline{n}^*)]$. Since $V$ is nondecreasing, 
$$V(\max(n, N)) = \max(V(n), V(N)) \le V(n) + V(N).$$ 
Taking expectation, we get 
\begin{equation}
    \label{eq:V-estimate-new}
\mathbb E[V(n\vee\overline{n}^*)] \le V(n) + \mathbb E[V(\overline{n}^*)] = V(n) + \mathbb E_{\overline{\pi}}[V].
\end{equation}
Combining~\eqref{eq:final-estimate} with~\eqref{eq:V-estimate-new}, we get the estimate in the statement of Lemma~\ref{lemma:1/t}. 
\end{proof}

\begin{remark}\label{rmk:slowest-general}
It is reasonable to expect that convergence of the joint process $(N,Z)$  will be at the slower of the two rates: the one of the birth-death process in itself and the one of the environment process. 
As we have shown in Theorem \ref{thm:polynomial}, the convergence rate for $X = (N, Z)$ is $1/t$, which is the slower of the rates for the two components. 
 It would be interesting 
to consider more general subexponential convergence rates $\psi(t)$, such as $t^{-\alpha}$ or $\exp\left[-c(\ln t)^{1 - \varepsilon}\right]$
for $\alpha, \varepsilon > 0$. 
For example, the birth-death process with arrival rates $1$ (independent of the position) and service rates $a^{\sqrt{n} - \sqrt{n-1}}$ for a constant $a > 1$ has polynomial rate of convergence $t^{-c}$ for any $c > 0$ (see  \cite[Example 1.9]{ergodic-degrees}).  
Such results would require proving that $\mathbb E[\psi(S_{\mathcal{J}})] < \infty$, which can be 
 harder than proving that $S_{\mathcal{J}}$ has finite mean or finite exponential moments.
It would be interesting to further investigate the polynomial rate of convergence for such models  in future work. 
 \hfill $\Box$
\end{remark}

\section{Appendix}
\label{sec-proofs}

\subsection{Appendix A: Proofs of Theorems  \ref{thm-BD} and \ref{thm-BD-diff}}
\begin{proof}[Proof of Theorem \ref{thm-BD}] 
The irreducibility and aperiodicity properties are straightforward. For the measure $\eta(n,z)$ in \eqref{eqn-BD-eta} to be finite, by Assumption \ref{as-Tn},
\begin{align*}
\sum_{(n,z)} \eta(n,z) = \sum_{(n,z)} r_n(z) v(z)  =\Xi  < \infty.
\end{align*}
To verify that $\eta(n,z)$ in \eqref{eqn-BD-eta} is an invariant measure, we prove that $\eta' \mathbf{R}=0$:  
\begin{align}
\label{eqn-eta-R}
\begin{split}
-\eta(n,z) & R[(n,z), (n,z)] \\ & = \eta(n-1,z) R[(n-1,z),(n,z)] + \eta(n+1,z) R[(n+1,z),(n,z)]   
\\ &  + \sum_{z'\neq z} \eta(n,z') R[(n,z'), (n,z)],  \quad  n=1,2,\dots,\quad  z \in \mathcal{Z};\\
-\eta(0,z) & R[(0,z),(0,z)] \\ & = \eta(1,z) R[(1,z),(0,z)] + \sum_{z'\neq z} \eta(0,z') R[(0,z'), (0,z)\,, n = 0\,, \quad z \in \mathcal Z.  
\end{split}
\end{align}
For~\eqref{eqn-eta-R} with $n \ge 1$, the left-hand side is
\begin{align*}
& \eta(n,z) \sum_{(n',z') \neq (n,z)} R[(n,z),(n',z')]  \\
&= \eta(n,z) \big( R[(n,z),(n+1,z)] + R[(n,z), (n-1,z)] + \sum_{z' \neq z} R[(n,z), (n,z')] \big) \\
&= r_n(z) v(z) \big( \lambda_n(z) + \mu_n(z) + \sum_{z'\neq z} r_n(z)^{-1} \tau_n(z,z') \big)  \\
&= r_n(z)  v(z) \big( \lambda_n(z) + \mu_n(z) \big) +  v(z) \sum_{z'\neq z} \tau_n(z,z'),
\end{align*}
and the right-hand side is equal to 
\begin{align*}
& r_{n-1}(z) v(z)\lambda_{n-1}(z)
+  r_{n+1}v(z)   \mu_{n+1}(z)  + \sum_{z'\neq z} r_n(z') v(z')  r_n(z')^{-1} \tau_n(z',z) \\
& =   v(z)\big( r_{n-1}(z) \lambda_{n-1}(z) +  r_{n+1}  \mu_{n+1}(z)  \big)   +  \sum_{z'\neq z} v(z') \tau_n(z',z). 
\end{align*}
We get equality thanks to the assumption in~\eqref{as-T-v} and the detailed balance equation in \eqref{eqn-N-balance-n}. For~\eqref{eqn-eta-R} with $n = 0$, the left-hand side is 
\begin{align*}
& \eta(0,z) \sum_{z'\neq z} R[(0,z),(0,z')]  = r_0(z) v(z) \big(R[(0,z),(1,z)]  + \sum_{z' \neq z} R[(0,z), (0,z')] \big) \\
& = r_0(z) v(z) \big(\lambda_0(z)   + \sum_{z' \neq z} r_0(z)^{-1} \tau_0(z,z')  \big)  = r_0(z) v(z) \lambda_0(z)   + v(z)  \sum_{z' \neq z} \tau_0(z,z')  \big),
\end{align*}
and the right-hand side is 
\begin{align*}
r_1(z) v(z)  \mu_1(z)  + \sum_{z'\neq z}  r_0(z')  v(z') r_0(z')^{-1} \tau_0(z',z) = r_1(z) v(z)  \mu_1(z)    + \sum_{z'\neq z}  v(z') \tau_0(z',z).
\end{align*}
This again leads to the equality thanks to \eqref{as-T-v} and \eqref{eqn-N-balance-0} for $n=0$. 
Thus we have shown that  $\pi(n,z)$ in \eqref{eqn-BD-pi} is an invariant probability measure. 
The positive recurrence property follows from  \cite[Theorem 3.5.3]{Norris} (see also \cite[Theorem 2.7.18]{SK}). 
The ergodicity property of convergence in total variation follows from  \cite{meyn1992stability}. 
 \end{proof}

\begin{proof}[Proof of Theorem \ref{thm-BD-diff}] 
The proof follows from an analogous argument as that of Theorem 3.1 in \cite{PSBS}, so we only highlight the differences. 
We apply \cite{kurtz2001stationary}, and use their notation as follows: let $E = \NN\times \mathcal{Z}$ and $U=\{0,1,\dots,m\}$, where ``0" indicates $\mathcal{Z}$ and $i=1,\dots,m$ for the faces $F_1,\dots, F_m$ of the boundary, and for $n\in \NN$, $z\in \mathcal{Z}$ and $u \in U$, 
\begin{align*}
\mu_0(\{u\}\times\{n\}\times dz) &= {\bf 1}_{u=0} r_n(z) \nu(dz),\quad 
\mu_1(\{u\}\times\{n\}\times dz) = {\bf 1}_{u\neq 0} r_n(z) \nu_i(dz),\\
\mu_0^E(\{n\}\times dz) &= r_n(z) \nu(dz),\quad \nu_1^E(\{n\}\times dz) = r_n(z) \big(\nu_{F_1}(dz) + \cdots + \nu_{F_m}(dz) \big),\\
\eta_0((n,z), \{u\}) &=  {\bf 1}_{u=0}, \quad 
\eta_1((n,z), \{u\}) =  {\bf 1}_{u\neq 0}, \\
Af((n,z),u)&:= \beta_n^{-1} r_n(z) \mathcal{L} f(n,z),\\
Bf((n,z),u) &:= {\bf 1}_{u\neq 0, z \in \partial D_i, i=1,\dots,m} \gamma_u(z) \cdot \nabla f(z).
\end{align*}
To check \cite[Condition 1.2]{kurtz2001stationary} on the absolutely continuous generator $A$ and the singular generator $B$, we can verify the conditions (i)-(v) in the same way as in the proof of \cite[Theorem 3.1]{PSBS}.  For the main condition in \cite[Theorem 1.7, (1.17)]{kurtz2001stationary}, we need to show that the generators $A$ and $B$ satisfy 
\begin{equation} \label{eqn-conditionKurtz}
\int_{E\times U} Af(x,u) \mu_0(dx\times du) + \int_{E\times U} Bf(x,u) \mu_1(dx\times du) =0.
\end{equation}
By the definitions of $A$ and $B$, we can write the left hand side as 
\begin{align*}
&\sum_{n=0}^\infty\int_{\mathcal{Z}}\beta_n^{-1} r_n(z) \mathcal{M}_zf(n,z) \nu(dz)\\ &  +\sum_{n=0}^\infty \bigg( \int_{\mathcal{Z}} \mathcal{A}f(n,z) \nu(dz) + \sum_{i=1}^m \int_{F_i} \gamma_i(z) \cdot \nabla f(z) \nu_{F_i} (dz)\bigg). 
\end{align*}
The sum of the last two terms is equal to zero, because the basic adjoint relationship holds for the reflected jump diffusion process $\widetilde{Z}$ (see, e.g.,  \cite{williams1995semimartingale}), that is, for $f\in C^2_b(\mathcal{Z})$, 
$$
\int_{\mathcal{Z}} \mathcal{A}f(n,z) \nu(dz) + \sum_{i=1}^m \int_{F_i} \gamma_i(z) \cdot \nabla f(z) \nu_{F_i} (dz) =0.
$$
For each $z\in\mathcal{Z}$, the birth-death process $N(t)$ has the stationary distribution given in \eqref{eqn-pi-n-BD} and \eqref{eqn-pi-0-BD}, which satisfy $r_n(z) \mathcal{M}_zf(n,\cdot) =0$ for each $z \in \mathcal{Z}$ and $n \in \ZZ$. Multiplying this by $\beta_n^{-1}$ and integrating over $z \in \mathcal Z$, we get: 
$$
\sum_{n=0}^\infty \int_{\mathcal{Z}}\beta_n^{-1} r_n(z) \mathcal{M}_zf(n,z) \nu(dz) = 0. 
$$
Thus we have verified that \eqref{eqn-conditionKurtz} holds. The rest of the proof follows the same argument as the proof of \cite[Theorem 3.1]{PSBS}.
\end{proof}

\subsection{Appendix B: A Comparison Lemma} 

\begin{lemma}[Lemma 5.1 in \cite{PSBS}]\label{lemma:tech}
 Fix constants $\alpha > 1$, $\beta, \gamma > 0$. Take two independent random variables $\xi \sim Exp(\beta)$ and $\eta > 0$  with
$\mathbb P(\eta > u) \le \alpha e^{-\gamma u}$ for $u \ge 0$. Then 
\begin{equation}
\label{eq:probab}
\mathbb P(\eta < \xi) \ge \alpha^{-\beta/\gamma}\frac{\gamma}{\beta + \gamma}\,.
\end{equation}
For $a \in [0, \beta + \gamma)$, the moment-generating function for $\xi\wedge\eta$ satisfies
\begin{equation}
\label{eq:MGF-est}
\mathbb E\big[e^{a(\xi\wedge\eta)}\big] \le \theta(\alpha, \beta, \gamma, a),
\end{equation}
where the function $\theta$ is defined in~\eqref{eq:theta}. 
\end{lemma}

\section*{Acknowledgments}
G. Pang was supported in part by NSF grant DMS-2216765. 
A. Sarantsev thanks Department of Mathematics \& Statistics at the University of Nevada, Reno, for the welcoming atmosphere 
for research. 
Y. Suhov thanks Department of Mathematics at the Pennsylvania State University for hospitality and support. Y. Suhov thanks
IHES, Bures-sur-Yvette, whose stimulating environment provides a constant source of inspiration.

\bibliographystyle{abbrv}
\bibliography{BD-IRE}

\begin{thebibliography}{10}

\bibitem{bacaer2014linear}
N.~Baca{\"e}r and A.~Ed-Darraz.
\newblock On linear birth-and-death processes in a random environment.
\newblock {\em Journal of Mathematical Biology}, 69(1):73--90, 2014.

\bibitem{browne1995piecewise}
S.~Browne and W.~Whitt.
\newblock Piecewise-linear diffusion processes.
\newblock {\em Advances in Queueing: Theory, Methods, and Open Problems},
  4:463--480, 1995.

\bibitem{Butkovsky}
O.~Butkovsky.
\newblock {Subgeometric rates of convergence of Markov processes in the
  Wasserstein metric}.
\newblock {\em Annals of Applied Probability}, 24(2):526 -- 552, 2014.

\bibitem{cogburn1980markov}
R.~Cogburn.
\newblock Markov chains in random environments: the case of {M}arkovian
  environments.
\newblock {\em Annals of Probability}, 8(5):908--916, 1980.

\bibitem{cogburn1981birth}
R.~Cogburn and W.~C. Torrez.
\newblock Birth and death processes with random environments in continuous
  time.
\newblock {\em Journal of Applied Probability}, 18(1):19--30, 1981.

\bibitem{cornez1987birth}
R.~Cornez.
\newblock Birth and death processes in random environments with feedback.
\newblock {\em Journal of Applied Probability}, 24(1):25--34, 1987.

\bibitem{das2016constructions}
A.~Das.
\newblock Constructions of {M}arkov processes in random environments which lead
  to a product form of the stationary measure.
\newblock {\em Markov Processes and Related Fields}, 23(2):211--232, 2017.

\bibitem{dieker2009reflected}
A.~Dieker and J.~Moriarty.
\newblock Reflected {B}rownian motion in a wedge: sum-of-exponential stationary
  densities.
\newblock {\em Electronic Communications in Probability}, 14:1--16, 2009.

\bibitem{Fort2009}
R.~Douc, G.~Fort, and A.~Guillin.
\newblock Subgeometric rates of convergence of f-ergodic strong markov
  processes.
\newblock {\em Stochastic Processes and their Applications}, 119(3):897 -- 923,
  2009.

\bibitem{economou2005generalized}
A.~Economou.
\newblock Generalized product-form stationary distributions for markov chains
  in random environments with queueing applications.
\newblock {\em Advances in Applied Probability}, 37(1):185--211, 2005.

\bibitem{Feller}
W.~Feller.
\newblock {\em {An Introduction to Probability Theory and its Applications}},
  volume~1.
\newblock Wiley, 1950.

\bibitem{gannon2016random}
M.~Gannon, E.~Pechersky, Y.~Suhov, and A.~Yambartsev.
\newblock Random walks in a queueing network environment.
\newblock {\em Journal of Applied Probability}, 53(2):448--462, 2016.

\bibitem{Fort2005}
F.~Gersende and G.~O. Roberts.
\newblock Subgeometric ergodicity of strong markov processes.
\newblock {\em Annals of Applied Probability}, 15(2):1565 -- 1589, 2005.

\bibitem{guillemin1995transient}
F.~Guillemin and A.~Simonian.
\newblock Transient characteristics of an {$M/M/\infty$} system.
\newblock {\em Advances in Applied Probability}, 27(3):862--888, 1995.

\bibitem{harrison1981distribution}
J.~M. Harrison and M.~I. Reiman.
\newblock On the distribution of multidimensional reflected {B}rownian motion.
\newblock {\em SIAM Journal on Applied Mathematics}, 41(2):345--361, 1981.

\bibitem{karatzas1991stochastic}
I.~Karatzas and S.~Shreve.
\newblock {\it Stochastic {C}alculus and {B}rownian {M}otion}, 1991.

\bibitem{kou2003modeling}
S.~C. Kou and S.~G. Kou.
\newblock Modeling growth stocks via birth-death processes.
\newblock {\em Advances in Applied Probability}, 35(3):641--664, 2003.

\bibitem{krenzler2015loss}
R.~Krenzler and H.~Daduna.
\newblock Loss systems in a random environment: steady state analysis.
\newblock {\em Queueing Systems}, 80(1):127--153, 2015.

\bibitem{krenzler2016jackson}
R.~Krenzler, H.~Daduna, and S.~Otten.
\newblock Jackson networks in nonautonomous random environments.
\newblock {\em Advances in Applied Probability}, 48(2):315--331, 2016.

\bibitem{kurtz2001stationary}
T.~Kurtz and R.~Stockbridge.
\newblock Stationary solutions and forward equations for controlled and
  singular martingale problems.
\newblock {\em Electronic Journal of Probability}, 6:1--52, 2001.

\bibitem{lindvall1979note}
T.~Lindvall.
\newblock A note on coupling of birth and death processes.
\newblock {\em Journal of Applied Probability}, 16(3):505--512, 1979.

\bibitem{lindvall}
T.~Lindvall.
\newblock {\em Lectures on the Coupling Method}.
\newblock Dover, 1992.

\bibitem{subgeometric-CTMC}
Y.~Liu, H.~Zhang, and Y.~Zhao.
\newblock Subgeometric ergodicity for continuous-time {M}arkov chains.
\newblock {\em Journal of Mathematical Analysis and Applications},
  368(1):178--189, 2010.

\bibitem{ergodic-degrees}
Y.~Mao.
\newblock Ergodic degrees for continuous-time {M}arkov chains.
\newblock {\em Science in China. Series A. Mathematics}, 47(2):161--174, 2004.

\bibitem{forward}
R.~R. Mazumdar and F.~M. Guillemin.
\newblock Forward equation for reflected diffusions with jumps.
\newblock {\em Applied Mathematics and Optimization}, 33:81--102, 1996.

\bibitem{meyn1992stability}
S.~P. Meyn and R.~L. Tweedie.
\newblock Stability of {M}arkovian processes {I}: {C}riteria for discrete-time
  chains.
\newblock {\em Advances in Applied Probability}, 24(3):542--574, 1992.

\bibitem{meyn1993stabilityII}
S.~P. Meyn and R.~L. Tweedie.
\newblock Stability of {M}arkovian processes {II}: {C}ontinuous-time processes
  and sampled chains.
\newblock {\em Advances in Applied Probability}, 25(3):487--517, 1993.

\bibitem{meyn1993stabilityIII}
S.~P. Meyn and R.~L. Tweedie.
\newblock Stability of {M}arkovian processes iii: {F}oster--{L}yapunov criteria
  for continuous-time processes.
\newblock {\em Advances in Applied Probability}, 25(3):518--548, 1993.

\bibitem{Norris}
J.~R. Norris.
\newblock {\em {Markov Chains}}.
\newblock Cambridge University Press, 1997.

\bibitem{otten2020queues}
S.~Otten, R.~Krenzler, H.~Daduna, and K.~Kruse.
\newblock Queues in a random environment.
\newblock {\em arXiv:2006.15712}, 2020.

\bibitem{PSBS}
G.~Pang, A.~Sarantsev, Y.~Belopolskaya, and Y.~Suhov.
\newblock Stationary distributions and convergence for {$M/M/1$} queues in
  interactive random environment.
\newblock {\em Queueing Systems}, 94(3):357--392, 2020.

\bibitem{prodhomme2021large}
A.~Prodhomme and {\'E}.~Strickler.
\newblock Large population asymptotics for a multitype stochastic sis epidemic
  model in randomly switched environment.
\newblock {\em arXiv:2107.05333}, 2021.

\bibitem{ross2019introduction}
S.~M. Ross.
\newblock {\em {\it Introduction to {P}robability {M}odels}}.
\newblock Academic Press, 12th edition, 2019.

\bibitem{sarantsev2016exp}
A.~Sarantsev.
\newblock Explicit rates of exponential convergence for reflected
  jump-diffusions on the half-line.
\newblock {\em ALEA Latin American Journal of Probability \& Mathematical
  Statistics}, 13:1069--1093, 2016.

\bibitem{Lithuania}
A.~Sarantsev.
\newblock Penalty method for obliquely reflected diffusions.
\newblock {\em Lithuanian Mathematical Journal}, 61:518--549, 2021.

\bibitem{sarantsev2021sub}
A.~Sarantsev.
\newblock Sub-exponential rate of convergence to equilibrium for processes on
  the half-line.
\newblock {\em Statistics \& Probability Letters}, 175:109115, 2021.

\bibitem{SK}
I.~M. Soukhov and M.~Kelbert.
\newblock {\em Probability and Statistics by Example: {M}arkov Chains: A Primer
  in Random Processes and their Applications}.
\newblock Cambridge University Press, 2008.

\bibitem{Busy}
W.~Stadie.
\newblock The busy period of the queueing system {$M/G/\infty$}.
\newblock {\em Journal of Applied Probability}, 22:697--704, 1985.

\bibitem{stroock1971diffusion}
D.~W. Stroock and S.~S. Varadhan.
\newblock Diffusion processes with boundary conditions.
\newblock {\em Communications on Pure and Applied Mathematics}, 24(2):147--225,
  1971.

\bibitem{tanaka1979stochastic}
H.~Tanaka.
\newblock Stochastic differential equations with reflecting boundary condition
  in convex regions.
\newblock {\em Hiroshima Mathematical Journal}, 9(1):163--177, 1979.

\bibitem{torrez1978birth}
W.~C. Torrez.
\newblock The birth and death chain in a random environment: instability and
  extinction theorems.
\newblock {\em Annals of Probability}, 6(6):1026--1043, 1978.

\bibitem{torrez1979calculating}
W.~C. Torrez.
\newblock Calculating extinction probabilities for the birth and death chain in
  a random environment.
\newblock {\em Journal of Applied Probability}, 16(4):709--720, 1979.

\bibitem{van1985conditions}
E.~A. Van~Doorn.
\newblock Conditions for exponential ergodicity and bounds for the decay
  parameter of a birth-death process.
\newblock {\em Advances in Applied Probability}, 17(3):514--530, 1985.

\bibitem{van2011rate}
E.~A. Van~Doorn.
\newblock Rate of convergence to stationarity of the system {$M/M/N/N+ R$}.
\newblock {\em Theory of Probability}, 19(2):336--350, 2011.

\bibitem{van2009speed}
E.~A. Van~Doorn and A.~I. Zeifman.
\newblock On the speed of convergence to stationarity of the {E}rlang loss
  system.
\newblock {\em Queueing Systems}, 63(1):241--252, 2009.

\bibitem{van2010bounds}
E.~A. van Doorn, A.~I. Zeifman, and T.~L. Panfilova.
\newblock Bounds and asymptotics for the rate of convergence of birth-death
  processes.
\newblock {\em Theory of Probability \& Its Applications}, 54(1):97--113, 2010.

\bibitem{ward2003properties}
A.~R. Ward and P.~W. Glynn.
\newblock Properties of the reflected {O}rnstein--{U}hlenbeck process.
\newblock {\em Queueing Systems}, 44(2):109--123, 2003.

\bibitem{williams1995semimartingale}
R.~J. Williams.
\newblock Semimartingale reflecting {B}rownian motions in the orthant.
\newblock {\em IMA Volumes in Mathematics and its Applications}, 71:125--137,
  1995.

\bibitem{zeifman1991some}
A.~I. Zeifman.
\newblock Some estimates of the rate of convergence for birth and death
  processes.
\newblock {\em Journal of Applied Probability}, 28(2):268--277, 1991.

\bibitem{zeifman1995upper}
A.~I. Zeifman.
\newblock Upper and lower bounds on the rate of convergence for nonhomogeneous
  birth and death processes.
\newblock {\em Stochastic Processes and their Applications}, 59(1):157--173,
  1995.

\bibitem{zeifman2017convergence}
A.~I. Zeifman and T.~L. Panfilova.
\newblock On convergence rate estimates for some birth and death processes.
\newblock {\em Journal of Mathematical Sciences}, 221(4):616--623, 2017.

\end{thebibliography}

\end{document}